\documentclass{amsart}

\usepackage{amsmath,amssymb,url,color,mathtools}
\usepackage[margin=3cm]{geometry}
\usepackage{tikz}

\newtheorem{lemma}{Lemma}[section]
\newtheorem{prop}[lemma]{Proposition}
\newtheorem{cor}[lemma]{Corollary}
\newtheorem{thm}[lemma]{Theorem}
\newtheorem{example}[lemma]{Example}
\newtheorem{thm?}[lemma]{Theorem?}

\newtheorem{remark}[lemma]{Remark}

\begin{document}
\title{Torsion Points and Galois Representations on CM Elliptic Curves}

\author{Abbey Bourdon}
\author{Pete L. Clark}

\newcommand{\etalchar}[1]{$^{#1}$}
\newcommand{\F}{\mathbb{F}}
\newcommand{\et}{\textrm{\'et}}
\newcommand{\ra}{\ensuremath{\rightarrow}}
\newcommand{\lra}{\ensuremath{\longrightarrow}}
\newcommand{\FF}{\F}
\newcommand{\ff}{\mathfrak{f}}
\newcommand{\Z}{\mathbb{Z}}
\newcommand{\N}{\mathcal{N}}
\newcommand{\ch}{}
\newcommand{\R}{\mathbb{R}}
\newcommand{\PP}{\mathbb{P}}
\newcommand{\pp}{\mathfrak{p}}
\newcommand{\C}{\mathbb{C}}
\newcommand{\Q}{\mathbb{Q}}
\newcommand{\ab}{\operatorname{ab}}
\newcommand{\Aut}{\operatorname{Aut}}
\newcommand{\gk}{\mathfrak{g}_K}
\newcommand{\gq}{\mathfrak{g}_{\Q}}
\newcommand{\OQ}{\overline{\Q}}
\newcommand{\Out}{\operatorname{Out}}
\newcommand{\End}{\operatorname{End}}
\newcommand{\Gal}{\operatorname{Gal}}
\newcommand{\CT}{(\mathcal{C},\mathcal{T})}
\newcommand{\lcm}{\operatorname{lcm}}
\newcommand{\Div}{\operatorname{Div}}
\newcommand{\OO}{\mathcal{O}}
\newcommand{\rank}{\operatorname{rank}}
\newcommand{\tors}{\operatorname{tors}}
\newcommand{\IM}{\operatorname{IM}}
\newcommand{\CM}{\mathbf{CM}}
\newcommand{\HS}{\mathbf{HS}}
\newcommand{\Frac}{\operatorname{Frac}}
\newcommand{\Pic}{\operatorname{Pic}}
\newcommand{\coker}{\operatorname{coker}}
\newcommand{\Cl}{\operatorname{Cl}}
\newcommand{\loc}{\operatorname{loc}}
\newcommand{\GL}{\operatorname{GL}}
\newcommand{\PGL}{\operatorname{PGL}}
\newcommand{\PSL}{\operatorname{PSL}}
\newcommand{\Frob}{\operatorname{Frob}}
\newcommand{\Hom}{\operatorname{Hom}}
\newcommand{\Coker}{\operatorname{\coker}}
\newcommand{\Ker}{\ker}
\newcommand{\g}{\mathfrak{g}}
\newcommand{\sep}{\operatorname{sep}}
\newcommand{\new}{\operatorname{new}}
\newcommand{\Ok}{\mathcal{O}_K}
\newcommand{\ord}{\operatorname{ord}}
\newcommand{\mm}{\mathfrak{m}}
\newcommand{\Ohell}{\OO_{\ell^{\infty}}}
\newcommand{\cc}{\mathfrak{c}}
\newcommand{\ann}{\operatorname{ann}}
\renewcommand{\tt}{\mathfrak{t}}
\renewcommand{\cc}{\mathfrak{a}}
\renewcommand{\aa}{\mathfrak{a}}
\newcommand\leg{\genfrac(){.4pt}{}}
\renewcommand{\gg}{\mathfrak{g}}
\renewcommand{\O}{\mathcal{O}}
\newcommand{\Spec}{\operatorname{Spec}}
\newcommand{\rr}{\mathfrak{r}}
\newcommand{\rad}{\operatorname{rad}}
\newcommand{\SL}{\operatorname{SL}}
\def\hh{\mathfrak{h}}

\begin{abstract}
{We prove several results on torsion points and Galois representations for complex multiplication (CM) elliptic curves over a number field containing the CM field. One result computes the degree in which such an elliptic curve has a rational point of order $N$, refining results of Silverberg  \cite{Silverberg88}, \cite{Silverberg92}. Another result bounds the size of the torsion subgroup of an elliptic curve with CM by a nonmaximal order in terms of the torsion subgroup of an elliptic curve with CM by the maximal order. Our techniques also yield a complete classification of both the possible torsion subgroups and the rational cyclic isogenies of a $K$-CM elliptic curve $E$ defined over $K(j(E))$. }
\end{abstract}

\maketitle

\tableofcontents

\section{Introduction}
\noindent
 Let $F$ be a field of characteristic $0$, and let $E_{/F}$ be an elliptic curve.  We say $E$ has \textbf{complex multiplication (CM)} 
if the endomorphism algebra 
\[ \End^0 E = \End (E_{/\overline{F}}) \otimes_{\Z} \Q \]
is strictly larger than $\Q$, in which case it is necessarily an imaginary quadratic field $K$ and $\OO \coloneqq \End (E_{/\overline{F}})$ is a $\Z$-order in $K$.  \\ \indent
{The general theory of complex multiplication has a long and rich history, with important contributions made by Kronecker, Weber, Fricke, Hasse, Deuring, and Shimura. For a summary of these foundational results, see \cite[Chapter 2]{SilvermanII}. More recent contributions to the study of torsion points and Galois representations on CM elliptic curves defined over number fields} have been made by Olson \cite{Olson74}, Silverberg 
\cite{Silverberg88}, \cite{Silverberg92}, Parish \cite{Parish89}, Aoki \cite{Aoki95}, \cite{Aoki06}, Ross \cite{Ross94}, Kwon \cite{Kwon99}, 
Prasad-Yogananda \cite{PY01}, Stevenhagen \cite{Stevenhagen01}, Breuer \cite{Breuer10}, Lombardo \cite{Lombardo15},  Lozano-Robledo \cite{LR}, Gaudron-R\'emond \cite{Gaudron-Remond18} and the present authors and our collaborators \cite{TORS1}, \cite{TORS2}, \cite{BCS}, \cite{CP15}, \cite{BCP}, \cite{BP16}.  In {this paper}, we consider the case of a CM elliptic curve 
defined over a number field that contains the CM field.  The case in which the ground field is a number field not assumed to contain 
the CM field is pursued in a followup paper \cite{BCII}.   There is related work of \'A. Lozano-Robledo done concurrently with 
the present work \cite{LR19}, which determines all possible images of the $\ell$-adic Galois representations of a CM elliptic 
curve $E$ over $\Q(j(E))$ up to conjugacy.  
\\ \\
Throughout this introduction we maintain the following notation: $K$ is an imaginary quadratic field, $\OO$ is an order in $K$, $\ff$ is the conductor of $\OO$, $K(\ff)$ is the $\ff$-ring class field of $K$ (i.e., $K(\ff) = K(j(E))$ for any $\OO$-CM elliptic curve $E$), $F$ 
is a number field containing $K$ and $N$ is a positive integer.

\subsection{The Torsion Degree Theorem} Let $\OO$ be an order in the imaginary quadratic field $K$, and let $N \in \Z^+$.  The following result was first proven by Silverberg \cite{Silverberg88}, \cite{Silverberg92} and then subsequently by Prasad-Yogananda \cite{PY01}.  

\begin{thm}[Silverberg]
\label{SILVERBERGTHM}
Let $F \supset K$ be a number field, and suppose that there is an $\OO$-CM elliptic curve $E_{/F}$ with an $F$-rational 
point of order $N$.  Then 
\[ \varphi(N) \leq \# \OO^{\times} \cdot [F:K]. \]
\end{thm}
\noindent
Theorem \ref{SILVERBERGTHM} is a crucial result in the study of torsion subgroups of CM elliptic curves over general 
number fields.  For instance, it was the main tool in the complete 
enumeration of torsion subgroups of CM elliptic curves defined over number fields of small degree \cite{TORS1}, \cite{TORS2}.  
\\ \indent
The hypotheses of Theorem \ref{SILVERBERGTHM} force $F \supset  K(\ff) = K(j(E))$.  Thus it is natural to define $T(\OO,N)$ 
to be the least degree $[F:K(\ff)$] of a number field $F \supset K$ over which some $\OO$-CM elliptic curve admits an $F$-rational point of order $N$.  We show in Theorem \ref{SPY} that 
\begin{equation}
\label{SPYINTRO}
 \varphi(N) \mid \# \OO^{\times} \cdot T(\OO,N),
\end{equation}
i.e., Theorem \ref{SILVERBERGTHM} holds as a divisibility. \\ \indent  Our first main result 
computes $T(\OO,N)$ in all cases and gives the analogous divisibility refinement.  

\begin{thm}
\label{BIGONE}
\label{BIGONE1}
Let $\OO$ be an order in the imaginary quadratic field $K$, and let $N$ be a positive integer.  There is an integer $T(\OO,N)$, explicitly 
computed in \S 7, such that: \\
(i) if $F \supset K$ is a number field and $E_{/F}$ is an $\OO$-CM elliptic curve with an $F$-rational point of order $N$, then $T(\OO,N) \mid [F:K(\ff)]$, and \\
(ii) there is a number field $F \supset K$ and an $\OO$-CM elliptic curve $E_{/F}$ such that $[F:K(\ff)] = T(\OO,N)$ and 
$E(F)$ has a point of order $N$.
\end{thm}
\noindent 
Equivalently, Theorem \ref{BIGONE} determines the least degree of a closed $\OO$-CM point on $X_1(N)_{/K}$ and shows that this degree divides the degree of all closed $\OO$-CM points.

\subsection{The Isogeny Torsion Theorem} A key feature of the present work is that we work with \emph{all} imaginary quadratic 
orders $\OO$, not just the maximal order $\OO_K$.  Working with nonmaximal orders entails certain technical 
complications.  For instance, if $F \supset K$ is a number field, then $E(F)[\tors]$ is a finite $\OO$-submodule of $E(\C)$.  
As we will see in \S 2.2, every finite $\OO$-submodule of $E(\C)$ is cyclic if and only if the order $\OO$ is maximal.  
\\ \indent
The phenomenon of ``ascending isogenies'' can sometimes be used to study $\OO$-CM elliptic curves 
in terms of $\OO_K$-CM elliptic curves, and this happens twice in the present paper.  
{Specifically,} let $E$ be an $\OO$-CM elliptic curve defined over a number field $F$, and let 
$\ff'$ be  a positive integer that divides $\ff$.  Then by \cite[Prop. 2.2]{BP16} there is an elliptic curve $(E_{\ff'})_{/F}$ 
such that $\OO(\ff') \coloneqq \End E_{\ff'}$ is the order of conductor $\ff'$ in $K$ and an $F$-rational isogeny $\iota_{\ff'} \colon E \ra E_{\ff'}$ that is cyclic of degree $\frac{\ff}{\ff'}$.  There is an embedding $F \hookrightarrow \mathbb{C}$ such that the base change of $\iota_{\ff'}$ to $\C$ is the natural map $\mathbb{C}/\OO \rightarrow \mathbb{C}/\OO(\ff')$ of complex elliptic curves.  The map $\iota_{\ff'}$ is universal for maps from an $\OO$-CM elliptic curve 
to an $\OO(\ff')$-CM elliptic curve \cite[\S 2.6]{BCII} and is thus unique, up to isomorphism on the target.  {Here is the first result making use of this canonical isogeny.}

\begin{thm}(Isogeny Torsion Theorem)
\label{TIT}
\label{ITT}
Let $\OO$ be an order in an imaginary quadratic field $K$, of conductor $\ff$, and let $\ff'$ be a positive integer dividing $\ff$.  
Let $F \supset K$ be a number field, and let $E_{/F}$ be an $\OO$-CM elliptic curve.  Let $\iota_{\ff'} \colon E \ra E_{\ff'}$ be 
the $F$-rational isogeny to an elliptic curve $E_{\ff'}$ with CM by the order in $K$ of conductor $\ff'$, as described above.  
Then we have
\[ \# E(F)[\tors] \mid \# E_{\ff'}(F)[\tors]. \]
\end{thm}
\noindent
In particular, taking $\ff'=1$, we see that $\# E(F)[\tors]$ is bounded by  $\# E_{1}(F)[\tors]$, where $(E_1)_{/F}$ is an $\OO_K$-CM elliptic curve.   We give examples where the exponent of $E_{\ff'}(F)[\tors]$ is strictly smaller than that of $E(F)[\tors]$, showing in general we cannot view $E(F)[\tors]$ as a subgroup of $E_{\ff'}(F)[\tors]$, and we prove that $\frac{\#E_{\ff'}(F)[\tors]}{\#E(F)[\tors]}$ can be arbitrarily large (see Propositions \ref{ABBEYPROP3} and \ref{ABBEYPROP2}). Moreover, the statement is false if we do not require $F \supset K$.  Theorem \ref{ITT} has applications to determining fields of moduli of partial level $N$ structures ($\S6.2$, $\S6.3$).

\subsection{The Reduced Galois Representation} There is a well-known interplay between points on modular curves over number fields and Galois representations of elliptic curves.  The proofs of Theorems \ref{BIGONE} and \ref{ITT} make use of Galois representations, 
and in the former case we build on a nearly complete description of the image of the mod $N$ Galois representation on an 
$\OO$-CM elliptic $E_{/K(\ff)}$.  
\\ \indent
For an elliptic curve $E$ defined over a number field $F$ and a positive integer $N$, 
the $\Z$-linear action of $\gg_F \coloneqq \Aut(\overline{F}/F)$ on $E[N]$ gives rise to the mod $N$ Galois representation:
\[ \rho_N \colon \gg_F \ra \GL_2(\Z/N\Z). \]
When $E$ does not have CM, a celebrated result of Serre \cite{Serre72} asserts that as $N$ varies over all positive integers, the index 
$[\GL_2(\Z/N\Z):\rho_N(\gg_F)]$ remains bounded.  This is certainly not the case when $E$ has CM: as usual, here we consider 
the case in which $F$ is a number field containing the CM field $K$.  Then for $N \in \Z^+$, Galois acts by $\OO$-linear endomorphisms of $E[N]$, {which is a free $\OO/N\OO$-module of rank one. Thus} the mod $N$ Galois representation 
takes the form
\[ \rho_N \colon \gg_F \ra (\OO/N\OO)^{\times} \hookrightarrow \GL_2(\Z/N\Z).\]
In the CM case the analogue of Serre's result is the boundedness of the index of $\rho_N$ in $(\OO/N\OO)^{\times}$ as $N$ varies.  In fact more is true: a slight variant of $\rho_N$ is surjective for all $\OO$ and $N$.  To motivate this, observe that fixing $\OO$ is the same as
 fixing $j(E)$ (up to Galois conjugacy), but fixing $j(E)$ does not determine the $K(\ff)$-rational model of $E$ and thus not 
$\rho_N$.  One gets from one model to another via a twist by $d \in K(\ff)^{\times}/K(\ff)^{\times \# \OO^{\times}}$.  If $E$ and $E^d$ are elliptic curves over $K(\ff)$ and $\rho_{E}, \rho_{E^d} \colon \gg_{K(\ff)} \ra (\OO/N\OO)^{\times}$ are their
mod $N$ Galois representations, then $\rho_{E^d} = \rho_E \otimes \chi_d$, where $\chi_d \colon \gg_{K(\ff)} \ra \OO^{\times}$ is the 
character corresponding to $d$.  Thus we define the \textbf{reduced mod N Cartan subgroup} 

\[
\overline{C_N(\OO)}=C_N(\OO)/q_N(\OO^{\times})
\]
to be the quotient of $C_N(\OO) = (\OO/N\OO)^{\times}$ by the image of $\OO^{\times}$ under the natural map $q_N \colon \OO \ra \OO/N\OO$ and the \textbf{reduced mod N Galois representation} to be the composite homomorphism

 \[ \overline{\rho_N}: \gg_{F} \stackrel{\rho_{N}}{\longrightarrow} C_N(\OO) \ra \overline{C_N(\OO)}.\]
{The key feature of $\overline{\rho_N}$ is that it} is independent of the $K(\ff)$-rational model.   \\ \indent
For an elliptic curve $E$ defined over a field $F$ of characteristic $0$, there is an $F$-rational isomorphism $\iota \colon E/\Aut(E) \stackrel{\sim}{\lra} \PP^1$, 
and a \textbf{Weber function} on $E$ is any function $\mathfrak{h} \colon E \ra \PP^1$ obtained by composing the quotient map 
with such an isomorphism $\iota$.  Then the field extension cut out by the reduced Galois representation is the field obtained by adjoining to $K(\ff)$ the values of the Weber function on the $N$-torsion points of $E$:
\[  \overline{\Q}^{\Ker \overline{\rho_N}} = K(\ff)(\mathfrak{h}(E[N])). \]
When $\OO = \OO_K$ the First Main Theorem of Complex Multiplication tells us that for any ideal $I$ of $\OO_K$ we have that
$K(\ff)(\mathfrak{h}(E[I]))$ is $K^{(I)}$, the $I$-ray class field of $K$.   It turns out that \[[K^{I}:K^{(1)}] = \# C_N(\OO) \]
and thus $\overline{\rho_N}$ is surjective.  The case of an arbitrary order is much less classical but still known: in \cite{Stevenhagen01}, {Stevenhagen} used Shimura's reciprocity law to show that for all $N \in \Z^+$, the Weber function field $K(\ff)(\mathfrak{h}(E[N]))$ is $K(\ff)^{N \OO_K}$, the \textbf{N-ray class field of $\OO$}: this is the finite abelian extension of $K$ corresponding to the image of the subgroup $\C^{\times} \times \{x \in \widehat{\OO}^{\times} \mid x \equiv 1 \pmod{N}\}$ in the norm one id\`ele class group of $K$.  Moreover, it follows from the adelic description \cite[p. 8]{Stevenhagen01} that 
\[ \Aut(K(\ff)^{N \OO_K}/K(\ff))= \overline{C_N(\OO)} \]
Thus we have the following result.

\begin{thm}[Stevenhagen]
\label{STEVENHAGEN}
\label{MAINTHM}
Let $\OO$ be an order in the imaginary quadratic field $K$, and let $N \in \Z^+$.  Then the reduced mod $N$ Galois representation 
\[ \overline{\rho_N}: \gg_{K(\ff)} \ra \overline{C_N(\OO)} \]
is surjective and $\overline{\Q}^{\Ker \overline{\rho_N}} = K(\ff)^{N \OO_K}$,
the $N$-ray class field of $\OO$.  
\end{thm}
\noindent
We will give a new proof of Theorem \ref{STEVENHAGEN}, as follows.  Let $E_{/K(\ff)}$ be a $\OO$-CM elliptic curve.  Using the canonical isogeny $\iota_1 \colon E \ra E_1$ to an $\OO_K$-CM elliptic curve, we show that the torsion field $K(\ff)(E[N])$ contains the ray class field $K^{N \OO_K}$: Theorem \ref{BCS3.16THM}b).  Using an observation of Parish we show that $K(\ff)(E[N])$ contains the ring class 
field $K(N\ff)$: Theorem \ref{PARISHTHM}.  By Theorem \ref{WFP}c)), {we have}
\[ K(\ff)(\mathfrak{h}(E[N])) \supset K^{N \OO_K}K(N\ff) = K(\ff)^{N \OO_K}, \]
{where the last equality can be shown using class field theory ($\S 5.1$). Since}
\[  [K(\ff)(\mathfrak{h}(E[N])):K(\ff)] \leq \# \overline{C_N(\OO)} = [K(\ff)^{N \OO_K}:K(\ff)], \]
we get 
\[ \overline{\Q}^{\Ker \overline{\rho_N}} = K(\ff)(\mathfrak{h}(E[N])) = K(\ff)^{N \OO_K}. \]
\noindent
Theorem \ref{STEVENHAGEN} has the following useful consequences:

\begin{cor}
\label{NEWCOR3} 
\label{1.1B}
\label{COR1.4}
For all number fields $F \supset K$ and all $\OO$-CM elliptic curves $E_{/F}$ we have \[
[C_N(\OO):\rho_N(\gg_F)] \mid \#\OO^{\times} [F:K(\ff)] \leq 6[F:K(\ff)].\]
\end{cor}

\begin{remark}
Corollary \ref{NEWCOR3} strengthens a result of D. Lombardo \cite[Thm. 1.5]{Lombardo15}, who showed that $[C_N(\OO):\rho_N(\gg_F)] \leq 6[F:K]$.

\end{remark}

\begin{cor}
\label{NEWCOR1}
\label{1.1C}
\label{COR1.2}
Let $N \in \Z^+$.  There is a number field $F \supset K$ 
and an $\OO$-CM elliptic curve $E_{/F}$ such that $E[N] = E[N](F)$ and $[F:K(j(E))] = \# \overline{C_N(\OO)}$.
\end{cor}

\begin{cor}
\label{LargeTwistCor}
For all $N \in \Z^+$, there is an $\OO$-CM elliptic curve $E_{/K(\ff)}$ such that $\rho_{E,N}(\gg_{K(\ff)}) = C_N(\OO)$.
\end{cor}
\noindent
Using Theorems \ref{BIGONE} and \ref{STEVENHAGEN} we also obtain a complete classification of the set of $N \in \Z^+$ such that an $\OO$-CM elliptic curve $E_{/K(\ff)}$ admits a $K(\ff)$-rational cyclic $N$-isogeny and of the possible torsion subgroups of 
an $\OO$-CM elliptic curve $E_{/K(\ff)}$. See $\S 6.6$ and $\S 6.7$.

\subsection{Acknowledgments} We are deeply indebted to R. Broker for making us aware of Stevenhagen's work \cite{Stevenhagen01}. We also thank the anonymous referee for helpful comments and suggestions.

\section{Preliminaries}

\subsection{Foundations} 
We begin by setting some terminology for orders in imaginary quadratic fields.  Let $K$ be an imaginary quadratic field, with ring of integers $\OO_K$, and let $w_K \coloneqq \# \OO_K^{\times}$ be the number of roots of unity in $K$.   Let $\OO$ be a $\Z$-order in $K$.  Let $\ff = [\OO_K:\OO]$ be the conductor of $\OO$, and let $\Delta$ be the discriminant of $\OO$. Then 
\[ \OO = \Z + \ff \OO_K, \ \Delta = \ff^2 \Delta_K. \]
For fixed $K$ and $\ff \in \Z^+$ there is a unique order $\OO(\ff)$ in $K$ of conductor $\ff$. Thus an imaginary quadratic order is determined by its discriminant $\Delta$, a negative integer which is $0$ or $1$ modulo $4$.  Conversely, for any negative integer $\Delta$ which is $0$ or $1$ modulo $4$, we put  \[\tau_{\Delta} = \frac{\Delta + \sqrt{\Delta}}{2}, \]
and then $\Z[\tau_{\Delta}]$ is an order in $K = \Q(\sqrt{\Delta})$ of discriminant $\Delta$.
\\ \\
Throughout this paper we will use the following terminological convention: by ``an order $\OO$'' we always mean a $\Z$-order 
$\OO$ in an imaginary quadratic field, which is determined as the fraction field of $\OO$ and denoted by $K$.  We may specify an order $\OO$ 
by giving its discriminant, which also determines $K$.  If $K$ is already given, then we specify an order $\OO$ in $K$ by giving 
the conductor $\ff$.
\\ \\
For any $\OO$-CM elliptic curve $E$ we have $K(j(E))=K(\ff)$, the ring class field of $K$ of conductor $\ff$ (\cite[Thm. 11.1]{Cox89}) Thus $[K(j(E)):K]$ is determined via the following formula.
\begin{thm}
\label{COX7.28THM}
For $N \in \Z^+$, let $K(N)$ denote the $N$-\textbf{ring class field} of $K$.  Then $K(1) = K^{(1)}$ is the Hilbert class 
field of $K$, and for all $N \geq 2$ we have 
\begin{equation*}
[K(N):K^{(1)}] = \frac{2}{w_K} N \prod_{p \mid N} \left(1 - \left(\frac{\Delta_K}{p} \right) \frac{1}{p} \right). 
\end{equation*}
\end{thm}
\begin{proof}
See e.g. \cite[Cor. 7.24]{Cox89}.
\end{proof}

\noindent For a number field $F$, $N \in \Z^+$ and $E_{/F}$ an elliptic curve, we denote by 
$\rho_N$ the homomorphism 
\[ \gg_F \ra \Aut E[N] \cong \GL_2(\Z/N\Z), \]
the \textbf{modulo N Galois representation}. If $E_{/F}$ has CM by the order $\OO$ in $K$, then $E[N] \cong_{\OO} \OO/N\OO$ -- i.e., 
we have an isomorphism of $\OO$-modules (see \cite[Lemma 1]{Parish89}, generalized in Lemma \ref{INVLEMMA1} below) -- and provided $F \supset K$ we have
\[
\rho_{N} \colon \gg_F \hookrightarrow \Aut_{\OO}E[N] \cong \GL_1(\OO/N\OO)=(\OO/N\OO)^{\times}.
\]
In other words, the image of the mod $N$ Galois representation lands in the  \textbf{mod N Cartan subgroup}
\[ C_N(\OO) = (\OO/N\OO)^{\times}. \]

\begin{lemma}
\label{LASTLEMMA}
Let $\OO$ be an order of discriminant $\Delta$, and let $N = p_1^{a_r} \cdots p_r^{a_r} \in \Z^+$.  \\
a) We have $C_N(\OO) = \prod_{i=1}^r C_{p_i^{a_i}}(\OO)$ (canonical isomorphism).  \\
b) We have $\# C_N(\OO) = N^2 \prod_{p \mid N} \left(1 - \left(\frac{\Delta}{p}\right) \frac{1}{p} \right)\left(1-\frac{1}{p} \right)$.
\end{lemma}
\begin{proof}
a) It suffices to tensor the Chinese remainder theorem isomorphism $\Z/N\Z = \prod_{i=1}^r \Z/p_i^{a_i} \Z$ 
with the $\Z$-module $\OO$ and pass to the unit groups. \\
b)  By \cite{TORS1}, for any prime number $p$ we have 
\[\# C_p(\OO) = p^2 \left(1-\left(\frac{\Delta}{p}\right) \frac{1}{p} \right)\left(1-\frac{1}{p}\right). \]
The natural map $C_{p^a}(\OO) \ra C_p(\OO)$ is surjective with kernel of  size $p^{2a-2}$  \cite[p. 3]{CP15}.  Together with 
part a) this shows that if $N = p_1^{a_1} \cdots p_r^{a_r}$ then 
\[ \# C_N(\OO) = \prod_{i=1}^r p_i^{2a_i-2} \left(p_i-1\right)\left(p_i - \left( \frac{\Delta}{p_i} \right) \right) = N^2 \prod_{p \mid N} \left(1 - \left(\frac{\Delta}{p}\right) \frac{1}{p} \right)\left(1-\frac{1}{p} \right).  \qedhere\]
\end{proof}

\subsection{Torsion Kernels}
Let $E_{/\C}$ be an $\OO$-CM elliptic curve.  For a nonzero ideal $I$ of $\OO$, we define the \textbf{I-torsion kernel}
\[ E[I] \coloneqq \{P \in E \mid \forall \alpha \in I, \ \alpha P = 0\}. \]
There is an invertible ideal $\Lambda \subset \OO$ such that 
\[ E \cong \C/\Lambda. \]
If we put 
\[ (\Lambda:I) \coloneqq \{ x \in \C \mid x I \subset \Lambda\} = \{x \in K \mid xI \subset \Lambda\} \]
then we have (immediately) that
\[ E[I] = \{ x \in \C/\Lambda \mid xI \subset \Lambda\} = (\Lambda:I)/\Lambda. \]
Let $ |I| \coloneqq \# \OO/I$.

\begin{lemma}
\label{LAZARUSLEMMA}
Let $I,J \subset \OO$ be nonzero ideals and $E_{/\C}$ be an $\OO$-CM elliptic curve.  \\
a) If $I \subset J$, then $E[J] \subset E[I]$.  \\
b) We have $E[I] \subset E[|I|]$.  In particular 
\[ \# E[I] \leq |I|^2. \]
\end{lemma}
\begin{proof}
a) This is immediate from the definition.  b) By Lagrange's Theorem, every element of $\OO/I$ is killed by $|I|$, so $|I| \in |I| \OO \subset I$.  Apply part a).
\end{proof}

\begin{lemma}
\label{INVLEMMA1}
If $I$ is an invertible $\OO$-ideal, then 
\[ E[I] = I^{-1} \Lambda / \Lambda \cong_{\OO} \OO/I. \]
In particular $\# E[I] = |I| = \# \OO/I$.
\end{lemma}
\begin{proof}
An ideal $I$ is invertible iff there is an $\OO$-submodule $I^{-1}$ of $K$ such that $I I^{-1} = \OO$.  If so, 
then for $x \in K$ we have 
\[x I \subset \Lambda \iff x I I^{-1} = x \OO \subset I^{-1} \Lambda \iff x \in I^{-1} \Lambda, \]
giving $E[I] = I^{-1} \Lambda/\Lambda$.  Because $\Lambda$ is a locally free $\OO$-module, for all $\pp \in \Spec \OO$
we have $\Lambda_{\pp} \cong \OO_{\pp}$ and thus $(I^{-1} \Lambda / \Lambda)_{\pp} \cong (I^{-1}/\OO)_{\pp} 
\cong (\OO/I)_{\pp}$.   Thus $I^{-1} \Lambda/ \Lambda$ is locally free of rank one as an $\OO/I$-module.  
But the ring $\OO/I$ is semilocal, hence has trivial Picard group: any locally free rank one $\OO/I$-module is 
isomorphic to $\OO/I$ \cite[Cor. 13.38]{Clark-CA}.
\end{proof}

\begin{lemma}
\label{DEDLEMMA}
Let $R$ be a Dedekind domain, and let $M$ be a cyclic torsion $R$-module, and let $N \subset M$ be an $R$-submodule.  Then: \\
a) $N$ is also a cyclic $R$-module.  \\
b) We have $N \cong R/\ann N$.
\end{lemma}
\begin{proof}
Let $I = \ann M$.  Since $M$ is a finitely generated torsion module over a domain, we have $I \neq 0$ and $M \cong R/I$.  
Thus $N \cong I'/I$ for some ideal $I' \supset I$.  The ring $R/I$ is principal Artinian \cite[Thm. 20.11]{Clark-CA}, so 
the ideal $I'/I$ of $R/I$ is principal.  Thus $N$ is a cyclic, torsion $R$-module, so $N \cong R/\ann N$.
\end{proof}

\begin{thm}
\label{Thm2.9}
Let $E_{/\C}$ be an $\OO_K$-CM elliptic curve, and let $M \subset E(\C)$ be a finite $\OO_K$-submodule.  Then $M = E[\ann M] \cong_{\OO_K} \OO_K/\! \ann M$ and thus $\# M = |\! \ann M|$. 
\end{thm}
\begin{proof}
That $M \subset E[\ann M]$ is a tautology.  Because $\OO_K$ is a Dedekind domain, every nonzero $\OO_K$-ideal is invertible, so Lemma \ref{INVLEMMA1} gives $\# E[\ann M] = |\! \ann M|$.  On the other hand, let $\mathfrak{t}=\#M$. Then $M \subset E[\mathfrak{t}] \cong_{\Ok} \Ok/\mathfrak{t}\Ok$, a finite cyclic $\Ok$-module. By Lemma \ref{DEDLEMMA} we have $M \cong \OO_K/\! \ann M$ so $\# M = |\! \ann M|$.  Thus $M = E[\ann M]$, so
Lemma \ref{INVLEMMA1} gives $M \cong \OO_K/\ann M$ and $\# M = |\! \ann M|$.  
\end{proof}
\begin{example}
\label{SUPERTECHNICALREMARK}
Theorem \ref{Thm2.9} fails for all nonmaximal orders. Indeed, let $\OO$ be a nonmaximal order in an imaginary quadratic field $K$.
There is nonzero prime ideal $\pp$ 
of $\OO$ such that the local ring $\OO_{\pp}$ is not a DVR.  If $\pp \cap \Z = (\ell)$, then $\OO/\pp \cong \Z/\ell\Z$.  
Since every ideal of $\OO$ can be generated by two elements, we have $\dim_{\OO/\pp} \pp/\pp^2 = 2$.  
Thus $\# \OO/\pp^2 = \ell^3$ and $(\ell^3) \subset \pp^2$.  It follows that in the quotient ring $\OO/\ell^3 \OO$ 
the maximal ideal $\pp + \ell^3 \OO$ is not principal.  Let $E_{/\C}$ be an $\OO$-CM elliptic curve, so 
$E[\ell^3] \cong_{\OO} \OO/\ell^3 \OO$.  So the $\OO$-submodule $M = \pp E[\ell^3]$ of $E[\ell^3]$ is 
not cyclic and thus not isomorphic to $\OO/\! \operatorname{ann} M$. 
\end{example}

\noindent
For a nonzero ideal $I$ of $\OO_K$, let $K^I$ denote the $I$-\textbf{ray class field} of $K$.

\begin{thm}(First Main Theorem of Complex Multiplication)
\label{FIRSTCM}
Let $E_{/\C}$ be an $\OO_K$-CM elliptic curve, and let $I$ be a nonzero ideal of $\OO_K$.  Let $\mathfrak{h} \colon E \ra \PP^1$ 
be a Weber function.  Then:
\[ K^{(1)}(\mathfrak{h}(E[I])) = K^I. \]
\end{thm}
\begin{proof}
See e.g. \cite[Thm. II.5.6]{SilvermanII}.
\end{proof}
\noindent
Combining Theorems \ref{Thm2.9} and \ref{FIRSTCM}, we get the class-field theoretic containment 
corresponding to any finite $\OO_K$-submodule of $E(\overline{F})$, for any $\OO_K$-CM elliptic curve 
$E$ defined over a number field $F \supset K$.

\noindent For convenience, we record here the formulas for $[K^I:K^{(1)}]$.  Here, $\varphi$ denotes Euler's totient function and $\varphi_K(I)$ the natural generalization for a nonzero ideal $I$ of $\OO_K$. That is,
\[ \varphi_K(I) = \# (\OO_K/I)^{\times} = |I| \prod_{\pp \mid I} \left(1-\frac{1}{|\pp|} \right).\]

\begin{lemma}
\label{CFTLEMMA}
Let $I$ be a nonzero ideal of $K$.  We put $U(K) = \OO_K^{\times}$ and 
$U_I(K) = \{x \in U(K) \mid x -1 \in I \}$.  \\
a) We have
\[ [K^{I}:K^{(1)}] = \frac{ \varphi_K(I)}{[U(K):U_I(K)]}. \]
b) If $K \neq \Q(\sqrt{-1}), \Q(\sqrt{-3})$, then 
\[ [K^{I}:K^{(1)}] = \begin{cases} \varphi_K(I) & I \mid (2) \\ 
\frac{ \varphi_K(I)}{2} & I \nmid (2) \end{cases}. \]
c) If $K = \Q(\sqrt{-1})$, then
\[ [K^{I}:K^{(1)}] = \begin{cases} \varphi_K(I) & I \mid (1+i) \\
\frac{ \varphi_K(I)}{2} & I \nmid (1+i) \text{ and } I \mid (2) \\
\frac{\varphi_K(I)}{4} & I \nmid (2)
\end{cases}. \]
d) If $K = \Q(\sqrt{-3})$, then 
\[ [K^I:K^{(1)}] = \begin{cases} 1 & I = (1) \\  
\frac{\varphi_K(I)}{2} & I \neq (1) \text{ and } I \mid (\zeta_3-1) \\
\frac{\varphi_K(I)}{3} & I = (2) \\
\frac{\varphi_K(I)}{6} & \text{ otherwise}
\end{cases}. \]
\end{lemma}

\begin{proof}
Parts b)-d) can be deduced from a), which appears as \cite[Cor. 3.2.4]{Cohen2}.
\end{proof}

\subsection{On Weber Functions} 
Let $F$ be a field of characteristic $0$.  For an elliptic curve $E_{/F}$, let $\mathfrak{h}: E \ra \PP^1$ be a Weber function for 
$E$ (cf. \S 1.3).

\begin{thm}(Weber Function Principle)
\label{WeberFunctionPrinciple}
\label{WFP}
Let $N \in \Z^{\geq 2}$, let $\OO$ be the order of conductor $\ff$ in $K$, and let $F=K(\ff)$. For an $\OO$-CM elliptic curve $E_{/F}$, fix an embedding $F \hookrightarrow \mathbb{C}$ such that $j(E)=j(\mathbb{C}/\OO)$. Define
\[ W(N,\OO) = K(\ff)(\mathfrak{h}(E[N])).\]
a) $W(N,\OO)$ is a subfield of $F(E[N])$ and $[F(E[N]):W(N,\OO) ]$ divides $\begin{cases} \# \OO^{\times} & N \geq 3 \\
\frac{\# \OO^{\times}}{2} & N = 2 \end{cases}$.  \\
b) There is an elliptic curve $E_{/F}$ such that 
\[ [F(E[N]):W(N,\OO) ] = 
\begin{cases} \# \OO^{\times} & N \geq 3 \\
\frac{\# \OO^{\times}}{2} & N = 2 \end{cases}. \]
c) As we range over all elliptic curves $E_{/F}$ with $j(E) = j(\C/\OO)$, we have 
\[ \bigcap_E F(E[N]) = W(N,\OO). \]
\end{thm}
\begin{proof} 
a) Let $w = \begin{cases} \# \OO^{\times} & N \geq 3 \\
\frac{\# \OO^{\times}}{2} & N = 2 \end{cases}$.   Let $\mu_w$ be the image of $\OO^{\times} \ra C_N(\OO)$, a cyclic group of 
order $w$.  The field $F(E[N])/F$ is Galois with Galois group $\rho_N(\gg_F) \subset C_N(\OO)$. Since $\mathfrak{h}(P)=\mathfrak{h}(Q)$ for points $P,Q$ on $E$ if and only if there is $\xi \in \OO^{\times}$ such that $\xi(P)=Q$ (e.g. \cite[Thm. I.7]{LangEll}), it follows that \[ W(N,\OO) = F(E[N])^{\rho_N(\gg_F) \cap \mu_w}. \] Thus \[ [F(E[N]):W(N,\OO) ] \mid w. \] 
b), c) If $E_{/F}$, $E'_{/F}$ with $j(E)=j(E')$, then $K(\ff)(\mathfrak{h}(E[N]))=K(\ff)(\mathfrak{h}(E'[N]))$ by the model independence of the Weber function. So  $W(N,\OO) \subset\bigcap_E F(E[N])$. To see that equality holds, let $E_{/F}$ have $j(E) = j(\C/\OO)$.  Let $\pp$ be a prime of $\OO_F$ that is unramified in $F' = F(E[N])$.  By weak 
approximation, there is $\pi \in \pp \setminus \pp^2$.  Put $L = F(\pi^{\frac{1}{w}})$, and let $\chi \colon \gg_F \ra \mu_w$ be a 
character with splitting field $\overline{F}^{\ker \chi} = L$.  (Explicitly, we may take $\chi(\sigma) = \frac{ \sigma(\pi^{1/w})}{\pi^{1/w}}$.)  Then $L/F$ is totally ramified over $\pp$, so $F'$ and $L$ are 
linearly disjoint over $F$.  It follows that 
\[ \rho_{N,E^{\chi}}(\gg_{F'}) = (\rho_{N,E_{/F'}} \otimes \chi)(\gg_{F'}) = \chi(\gg_{F'}) = \mu_w. \]
Thus
\[ w = [F(E^{\chi}[N]):F(E[N]) \cap F(E^{\chi}[N])] \mid [F(E^{\chi}[N]):W(N,\OO) ] \mid w, \]
so $F(E^{\chi}[N])$ has degree $w$ over $W(N,\OO)  = F(E[N]) \cap F(E^{\chi}[N])$.
\end{proof}

\subsection{A Containment of Weber Function Fields}

\begin{thm}
\label{BCS3.16THM}
Let $K$ be an imaginary quadratic field, and let $\OO \subset \OO'$ be orders in $K$. \\
a)  For all $N \in \Z^+$ we have a containment of Weber function fields $W(N,\OO) \supset W(N,\OO')$. \\
b) If $E$ is a $K$-CM elliptic curve defined over a number field $F \supset K$, then we have 
\begin{equation}
\label{BCS3.16EQ}
F(\mathfrak{h}(E[N])) \supset K^{N \OO_K}.
\end{equation}
\end{thm}
\begin{proof}
a) Let $\ff$ (resp. $\ff'$) be the conductor of $\OO$ (resp. of $\OO'$), so $\ff' \mid \ff$.  Let $E_{/K(\ff)}$ be an $\OO$-CM elliptic curve.  
Let $\iota = \iota_{\ff,\ff'} \colon E \ra E'$ be the canonical cyclic $\frac{\ff}{\ff'}$-isogeny to an $\OO'$-CM elliptic curve $(E')_{/K(\ff)}$.  
There is an embedding $K(\ff) \hookrightarrow \C$ such that $E(\C) \cong E/\OO$ and $(E')(\C) \cong E/\OO'$ and $\iota \colon E/\OO \ra E/\OO'$ is the natural map.  Then $\iota$ maps $\frac{1}{N} + \OO$ to $\frac{1}{N} + \OO'$, which generates $E'[N]$ as an $\OO'$-module.  Thus
\[ K(\ff)(E[N]) \supset K(\ff)(E'[N]) \supset K(\ff')(E'[N]) \supset W(N,\OO'), \]
and Theorem \ref{WFP}c) gives 
\[ W(N,\OO) \supset W(N,\OO'). \]
b) This follows from part a) and Theorem \ref{FIRSTCM}.
\end{proof}
\noindent
Though Theorem \ref{BCS3.16THM} is a consequence of Theorem \ref{STEVENHAGEN}, we prove it here as an ingredient in our new proof of Theorem \ref{STEVENHAGEN}.   Part b) had been proved earlier in 
 \cite[Thm. 3.16]{BCS}, but the present argument seems more transparent.

\section{Proof of the Isogeny Torsion Theorem}
\noindent
For a quadratic field $K$ and $d \in \Z^+$ we will write $\OO(d)$ for the unique order in $K$ of conductor $d$.  We recall the setup: let $F \supset K$ be a number field, let $E_{/F}$ be an elliptic curve with endomorphism ring an order $\OO$ 
of conductor $\ff$ in $K$, and let $\ff'$ be a positive integer that divides $\ff$.  Then there is a canonical $F$-rational cyclic $\frac{\ff}{\ff'}$-isogeny 
$\iota_{/\ff'}: E \ra E_{\ff'}$ such that $\End E_{\ff'} = \OO(\ff')$.  
\\ \\
There is a field embedding $F \hookrightarrow \C$ such that $E(\C) \cong \C/\OO(\ff)$, $E_{\ff'}(\C) \cong \C/\OO(\ff')$ 
and $(\iota_{\ff'})_{/\C}$ is the quotient map $\C/\OO(\ff) \ra \C/\OO(\ff')$.  Put $\tau_K = \frac{\Delta_K + \sqrt{\Delta_K}}{2}$, so
$\OO(\ff) = \Z[\ff \tau_K]$ and $\OO(\ff') = \Z[\ff' \tau_K]$.  For $N \in \Z^+$, let
\[ e_{1,\ff}(N) \coloneqq \frac{1}{N} + \OO(\ff), \ e_{2,\ff}(N) \coloneqq \frac{\ff \tau_K}{N} + \OO(\ff), \]
\[ e_{1,\ff'}(N) \coloneqq \frac{1}{N} + \OO(\ff'), \ e_{2,\ff'}(N) \coloneqq \frac{\ff' \tau_K}{N} + \OO(\ff'). \]
Then $\Ker (E[N] \stackrel{\iota_{\ff'}}{\ra} E_{\ff'}[N])$ is cyclic of order $\gcd(N,\frac{\ff}{\ff'})$, 
generated by $\frac{N}{\gcd(N,\frac{\ff}{\ff'})} e_{2,\ff}(N)$, and $\iota_{\ff'}(e_{1,\ff}(N)) = e_{1,\ff'}(N)$. \\ \indent
For finite commutative groups $T_1$ and $T_2$, we have $\# T_1 \mid \# T_2$ if and only if $ \# T_1[\ell^{\infty}] \mid \# T_2[\ell^{\infty}]$ 
for all prime numbers $\ell$.  So it suffices to show: for all prime numbers $\ell$, we have $\# E(F)[\ell^{\infty}] \mid 
\# E_{\ff'}(F)[\ell^{\infty}]$.  Write $\ff = \ell^{c_1} \overline{\ff}$ with $\gcd(\overline{\ff},\ell) = 1$ and $\ff' = 
\ell^{c_2} \overline{\ff'}$ with $\gcd(\overline{\ff'},\ell) = 1$.  Then we have 
\[ \# E(F)[\ell^{\infty}] = \# E_{\ell^{c_1}}(F)[\ell^{\infty}], \ \# E_{\ff'}(F)[\ell^{\infty}] = \# E_{\ell^{c_2}}(F)[\ell^{\infty}], \]
so we may assume that $\ff = \ell^{c_1}$ and $\ff' = \ell^{c_2}$.  Indeed, it is enough to treat the case $c_2 = c_1 -1$, 
since repeated application of this case yields the general case.  So suppose $\ff = \ell^c$ for some $c \in \Z^+$ and $\ff' = \ell^{c-1}$.  By (e.g.) 
the Mordell-Weil Theorem, there are integers $0 \leq a \leq b$ such that 
\[ E(F)[\ell^{\infty}] \cong \Z/\ell^a \Z \oplus \Z/\ell^b \Z. \]
Let $Q \coloneqq \frac{1}{\ell^a} + \OO(\ff) \in E(F)$.  Since $\iota_{\ff'}$ is $F$-rational, we have that $Q' \coloneqq \iota_{\ff'}(Q) = \frac{1}{\ell^a} + \OO(\ff')$ lies in $E_{\ff'}(F)$ and generates $E_{\ff'}[\ell^a]$ as an $\OO(\ff')$-module, so $E_{\ff'}[\ell^a] = E_{\ff'}(F)[\ell^a]$.  If $a = b$, it follows that $\# E(F)[\ell^{\infty}] 
\mid \# E_{\ff'}(F)[\ell^{\infty}]$, so we may assume $b > a$. Since 
$\Ker (E[\ell^{\infty}] \stackrel{\iota_{\ff'}}{\ra} E_{\ff'}[\ell^{\infty}])$ has order $\ell$, we have $\Z/\ell^a \Z \oplus \Z/\ell^{b-1} \Z \hookrightarrow E_{\ff'}(F)[\ell^{\infty}]$. Thus it suffices to show that $E_{\ff'}(F)$ has either a point of order $\ell^b$ or has full $\ell^{a+1}$-torsion.  
\\ \indent
Let $P = E(F)$ be a point of order $\ell^b$, and write $P = \alpha e_{1,\ff}(\ell^b) + \beta e_{2,\ff}(\ell^b)$ with $\alpha,\beta \in \Z/\ell^b \Z$.
If $\ell \nmid \alpha$ then since $\ff = \ell \ff'$ we have that $\iota_{\ff'}(P) = \alpha e_{1,\ff'}(\ell^b) + \ell \beta e_{2,\ff'}(\ell^b)$ has order $\ell^b$ and we are done, so we may assume that $\ell \mid \alpha$, in which case $\ell \nmid \beta$.  With respect to the basis $e_{1,\ff}(\ell^b), e_{2,\ff}(\ell^b)$ 
of $E[\ell^b]$, the image of the mod $\ell^b$ Galois representation on $E$ consists of matrices of the form 
\begin{equation}
\label{ITTEQ1}
\left[ \begin{array}{cc} a & b \ell^{2c} \frac{\Delta_K-\Delta_K^2}{4} \\ b & a + b \ell^c \Delta_K \end{array} \right] \text{ with }a,b \in \Z/\ell^b\Z.
\end{equation}
Since $E(F)$ has full $\ell^a$-torsion, we have $a \equiv 1 \pmod{\ell^{a}}$ and $b \equiv 0 \pmod{\ell^{a}}$. Thus
\[ \rho_{\ell^{a+1}}(\gg_F) \subset \left\{\left[ \begin{array}{cc} 1 + \ell^a A & 0 \\ \ell^a B & 1 + \ell^a A \end{array} \right] \bigg{\vert} A,B \in \Z/\ell^{a+1}\Z\right\}. \] 
Since $\ell^{b-a-1}P =  \alpha e_{1,\ff}(\ell^{a+1}) + \beta e_{2,\ff}(\ell^{a+1})$ is $F$-rational, all 
such matrices in the image of Galois satisfy
\[ \left[ \begin{array}{cc} 1 + \ell^a A & 0 \\ \ell^a B & 1 + \ell^a A \end{array} \right] \left[ \begin{array}{c} 
 \alpha \\ 
\beta \end{array} \right] = \left[ \begin{array}{c} \alpha \\ 
\beta \end{array} \right], \]
and thus $\ell^a \alpha B + \beta + \ell^a A \beta \equiv \beta \pmod{\ell^{a+1}}$. Since $\ell \mid \alpha$, we get  
\[ \ell^a A \beta \equiv -\ell^a \alpha B \equiv 0 \pmod{\ell^{a+1}}, \]
and thus $\ell \mid A$ and $\rho_{\ell^{a+1}}(\gg_F)$ consists of matrices of the form $\left[ \begin{array}{cc} 1 & 0 \\ \ell^a B & 1\end{array} \right]$ for $B \in \Z/\ell^{a+1}\Z$.  It follows that for all $\sigma \in \gg_F$, there is $B \in \Z/\ell^{a+1} \Z$ such that 
\begin{align*}
 \sigma(\iota_{\ff'}(e_{1,\ff}(\ell^{a+1}))) &= \iota_{\ff'}( e_{1,\ff}(\ell^{a+1}) + B \ell^a \iota_{\ff'}(e_{2,\ff}(\ell^{a+1})) \\
 &=
\iota_{\ff'}(e_{1,\ff}(\ell^{a+1})) +  B \ell^a (\ell e_{2,\ff'}(\ell^{a+1}))  
= \iota_{\ff'}(e_{1,\ff}(\ell^{a+1})). 
\end{align*}
Thus $e_{1,\ff'}(\ell^{a+1}) = \iota_{\ff'}(e_{1,\ff}(\ell^{a+1})) \in E_{\ff'}(F)$.  Since the $\OO(\ff')$-submodule generated 
by $e_{1,\ff'}(\ell^{a+1})$ is $E_{\ff'}[\ell^{a+1}]$, we get $\Z/\ell^{a+1}\Z \oplus \Z/\ell^{a+1}\Z \hookrightarrow E_{\ff'}(F)$, 
completing the proof of Theorem \ref{ITT}.

\section{The Projective Torsion Point Field}
\noindent
Let $F$ be a field. For a positive integer $N$ not divisible by the characteristic of $F$ and $E_{/F}$ an elliptic curve, we define the \textbf{projective modulo N Galois representation} as the composite map
\[\PP \rho_N:  \gg_F \xrightarrow{\rho_N} \Aut E[N] \cong \GL_2(\Z/N\Z) \ra \PGL_2(\Z/N\Z) \coloneqq \GL_2(\Z/N\Z)/(\Z/N\Z)^{\times}.\]
The \textbf{projective torsion field} is 
\[ F( \PP E[N]) = \overline{F}^{\Ker \PP \rho_N}. \]
Thus $F( \PP E[N])$ is the unique minimal field extension of $F$ on which the image of $\rho_N$ consists of scalar matrices. It follows that $F(E[N])/F( \PP E[N])$ is a Galois extension with automorphism group a subgroup of $(\Z/N\Z)^{\times}$.

Observe that the projective Galois representation and thus the projective torsion field are unchanged by \emph{quadratic} twists. If $E_{/F}$ has CM by an order of discriminant $\Delta = \ff^2 \Delta_K \neq -3,-4$ and $F \supset K$, then the projective $N$-torsion field is a well-defined abelian extension of $K(\ff)$.  
When 
$\Delta = -4$ (resp. $\Delta = -3$) we have quartic twists (resp. sextic twists) which can change the projective Galois 
representation and the projective torsion field.  

\begin{thm}
\label{PARISHTHM}
Let $\OO$ be an order of discriminant $\Delta = \ff^2 \Delta_K$.  Let $E$ be an $\OO$-CM elliptic curve defined 
over $F = K(\ff)$.  Let $N \geq 2$.  \\
a) (Parish \cite{Parish89}) We have $F(\PP E[N]) \supset K(N\ff)$.  Thus we may put \[d(E,N) = [F(\PP E[N]):K(N\ff)].  \]
b) (Parish \cite{Parish89}) If $\Delta \notin \{-3,-4\}$, then $d(E,N) = 1$, i.e., $F(\PP E[N]) = K(N \ff)$.  \\
c) If $\Delta = -4$, then $d(E,N) \mid 2$.  \\
d) If $\Delta = -3$, then $d(E,N) \mid 3$. 
\end{thm}
\begin{proof}
For $N \in \Z^+$, let $\OO(N)$ be the order of conductor $N$ in $K$.  Thus $\OO = \OO(\ff)$. \\
Step 1: We show that $F(\PP E[N]) \supset K(N\ff)$ in all cases.  \\
There is a field embedding $F \hookrightarrow \C$ such that $E(\C) \cong \C/\OO$.  The $\C$-linear map $z \mapsto Nz$ carries 
$\OO(\ff)$ into $\OO(N\ff)$ and induces a cyclic $N$-isogeny $\C/\OO(\ff) \ra \C/\OO(N\ff)$.  Let $C$ be the kernel of this 
isogeny, viewed as a finite \'etale subgroup scheme of $E_{/\C}$.  Then $C$ has a (unique) minimal field of definition $F(C) \subset 
F(E[N])$, hence of finite degree over $F$.  The field $F(\PP E[N])$ is precisely the compositum of the minimal fields of 
definition of all order $N$ cyclic subgroup schemes $C \subset E_{/\C}$, so $F(C) \subset F(\PP E[N])$.  Since $C$ is $F(\PP E[N])$-rational, the elliptic curve $E/C$ has a model over this field, and thus
\[ F(\PP E[N]) \supset K(j(E/C)) = K(N\ff). \]
Step 2: In view of Step 1, we have $F(\PP E[N]) \supset K(N\ff) \supset K(\ff) = K(j(E))$, so we have $F(\PP E[N]) = K(N\ff)$ iff $[F(\PP E[N]):K(\ff)] \leq [K(N\ff):K(\ff)]$.  We have
\[ [F(\PP E[N]):K(\ff)] = \# \PP \rho_N(\gg_F) \leq \# (\OO/N\OO)^{\times}/(\Z/N\Z)^{\times} = N \prod_{p \mid N} \left(1-
\left(\frac{\Delta}{p} \right) \frac{1}{p} \right). \]
$\bullet$ Suppose $\ff > 1$.  Using Theorem \ref{COX7.28THM} to compute $[K(N\ff):K^{(1)}]$ and $[K(\ff):K^{(1)}]$ gives 
\[ [K(N\ff):K(\ff)] = \frac{ [K(N\ff):K^{(1)}]}{[K(\ff):K^{(1)}]} = N \prod_{p \mid N, \ p \nmid \ff} \left(1-\left(\frac{\Delta}{p}\right) 
\frac{1}{p} \right)  = N \prod_{p \mid N} \left(1-
\left(\frac{\Delta}{p} \right) \frac{1}{p} \right), \]
because $1- \left( \frac{\Delta}{p} \right) \frac{1}{p} = 1$ for all $p \mid \ff$.  Thus $d(E,N)  = 1$ in this case.  \\
$\bullet$ Suppose $\ff = 1$, so $\Delta = \Delta_K$. Then 
\[ [K(N\ff):K(\ff)] = [K(N):K^{(1)}] = \frac{2}{w_K} N \prod_{p \mid N} \left(1-\left(\frac{\Delta}{p} \right) \frac{1}{p} \right). \]
If $\Delta \notin \{-3,-4\}$ then $\frac{2}{w_K} = 1$, and again we get $d(E,N) = 1$.  If $\Delta = -4$ then 
$\frac{2}{w_K} = \frac{1}{2}$, so the calculation shows $d(E,N) \in \{1,2\}$, and if $\Delta = -3$ then $\frac{2}{w_K} = \frac{1}{3}$, 
so the calculation shows $d(E,N) \in \{1,3\}$.  
\end{proof}

\begin{remark}
a) Theorem \ref{PARISHTHM}a) and Theorem \ref{PARISHTHM}b)  are due to Parish \cite[Prop. 3]{Parish89}.  However, Parish alludes to a calculation of the above sort rather than explicitly carrying it out.  Since Theorem \ref{PARISHTHM} 
will play an important role in the proof of Theorem \ref{MAINTHM}, we have given a complete proof. \\ 
b)
In \cite[Prop. 3]{Parish89}, Parish assumes $K \neq \Q(\sqrt{-1}), \Q(\sqrt{-3})$.  Later on  \cite[p. 263]{Parish89}, he claims: \\
$\bullet$ If $\Delta = -4$ then $F(\PP E[N]) = K(N)$ for all $N \geq 3$, and \\
$\bullet$ If $\Delta = -3$ then $F(\PP E[N]) = K(N)$ for all $N \geq 4$. \\
As we will see shortly in Example \ref{DENEXAMPLE}, both claims are false.
\end{remark}

\noindent
The following result is an an analogue of \cite[Thm. 5.6]{BCS} for higher twists.

\begin{prop}(Higher Twisting at the Bottom)
\label{HIGHERBOTTOMPROP}
 \\
For $M \in \Z^+$, we denote the mod $M$ cyclotomic character by $\chi_M$. \\
a) Let $K = \Q(\sqrt{-1})$ and let $\ell \equiv 5 \pmod{8}$ be a prime number.  There is a character $\Psi \colon \gg_K \ra (\Z/\ell\Z)^{\times}$ of order $\frac{\ell-1}{4}$ and an $\OO_K$-CM elliptic curve $E_{/K}$ such that the 
mod $\ell$ Galois representation is 
\[ \sigma \mapsto \rho_{\ell}(\sigma) = \left[ \begin{array}{cc} \Psi(\sigma) & 0 \\ 0 & \Psi^{-1}(\sigma) \chi_\ell(\sigma) \end{array} \right]. \]
b) Let $K = \Q(\sqrt{-3})$ and let $\ell \equiv 7,31 \pmod{36}$ be a prime number.  There is a character $\Psi \colon \gg_K \ra (\Z/\ell\Z)^{\times}$ of order $\frac{\ell-1}{6}$ and an $\OO_K$-CM elliptic curve $E_{/K}$ such that the 
mod $\ell$ Galois representation is 
\[ \sigma \mapsto \rho_{\ell}(\sigma) = \left[ \begin{array}{cc} \Psi(\sigma) & 0 \\ 0 & \Psi^{-1}(\sigma) \chi_\ell(\sigma) \end{array} \right]. \]
\end{prop}
\begin{proof}
a) Let $(E_1)_{/K}$ be an $\OO_K$-CM elliptic curve.  Because $\ell \equiv 1 \pmod{4}$ the Cartan subgroup $C_{\ell}(\OO_K)$ is split. It follows that there are precisely two $C_{\ell}(\OO_K)$-stable one-dimensional $\Z/\ell\Z$-subspaces of $E_1[\ell]$, so we may take basis vectors 
$e_1$ and $e_2$ for $E_1[\ell]$ lying in these two subspaces.   For such a basis, the mod $\ell$ Galois representation has the form
\[ \sigma \mapsto \rho_{\ell}(\sigma) = \left[ \begin{array}{cc} \Psi_1(\sigma) & 0 \\ 0 & \Psi_1^{-1}(\sigma) \chi_\ell(\sigma) \end{array} \right] \]
for a character $\Psi_1 \colon \gg_K \ra (\Z/\ell\Z)^{\times}$.  Under this isomorphism, the matrix representation of $i \in \OO_K$ is a 
diagonal matrix $\left[ \begin{array}{cc} z & 0 \\ 0 & z^{-1} \end{array} \right]$, where $z$ is a primitive $4$th root of unity in $\Z/\ell\Z$.  A general $\OO_K$-CM elliptic curve over $K$ is of the form $E_1^{\psi}$ for a character $\psi \colon \gg_K \ra \mu_4 \subset (\Z/\ell\Z)^{\times}$.  Let $Q_4(\ell) = (\Z/\ell\Z)^{\times}/(\Z/\ell \Z)^{\times 4}$.  Then the image of $z$ in $Q_4(\ell)$ has order $4$: if not, there 
is $w \in (\Z/\ell\Z)^{\times}$ such that $z = w^2$, and then $w$ has order $8$ in $(\Z/\ell \Z)^{\times}$, contradicting the 
assumption that $\ell \equiv 5 \pmod{8}$.  Thus the natural map $\mu_4 \ra Q_4(\ell)$ given by $i \mapsto z \pmod{(\Z/\ell\Z)^{\times 4}}$ is an isomorphism; 
we denote the inverse isomorphism $Q_4(\ell) \ra \mu_4$ by $\iota$.  Now take  
\[ \psi \colon \gg_K \stackrel{\Psi_1^{-1}}{\ra}  (\Z/\ell\Z)^{\times} \stackrel{q}{\ra} Q_4(\ell) \stackrel{\iota}{\ra} \mu_4;\]
here $q$ is the quotient map.
Let $\Psi_2 = \psi \Psi_1$.  Then the twist $E_1^{\psi}$ has mod $\ell$ Galois representation 
\[ \sigma \mapsto \rho_{\ell}(\sigma) = \left[ \begin{array}{cc} \Psi_2(\sigma) & 0 \\ 0 & \Psi_2^{-1}(\sigma) \chi_\ell(\sigma) \end{array} \right]. \]
The composite $\Psi_2 \colon \gg_K \ra (\Z/\ell\Z)^{\times} \ra  Q_4(\ell)$ is trivial, so $\Psi_2(\gg_K)$ has order $c \mid \frac{\ell-1}{4}$. Thus 
\[\# \rho_{\ell,E_1^{\psi}}(\gg_K) \mid c(\ell-1) \mid \frac{ (\ell-1)^2}{4} = [K^{(\ell)}:K^{(1)}] = [K^{(\ell)}:K], \]
where the last equality holds since $K$ has class number 1. Because $K(E_1^{\psi}[\ell]) \supset K^{(\ell)}$, we have $\# \rho_{\ell,E_1^{\psi}}(\gg_K) = \frac{(\ell-1)^2}{4}$ and 
$c = \frac{\ell-1}{4}$.  \\
b) Since $\ell \equiv 1 \pmod{3}$, we have a primitive $6$th root of unity $z$ in $\Z/\ell\Z$.  
Since $\ell \equiv 7,  31 \pmod{36}$, we have $4,9 \nmid \ell-1$, so $z$ has order $6$ in $Q_6(\ell) = (\Z/\ell\Z)^{\times}/(\Z/\ell \Z)^{\times 6}$.   Also $\frac{(\ell-1)^2}{6} = [K^{(\ell)}:K^{(1)}]$.  The argument of part a) carries over.
\end{proof}

\begin{example}
\label{DENEXAMPLE}
a) Let $K = \Q(\sqrt{-1})$, and let $\ell \equiv 5 \pmod{8}$.  Let $E_{/K}$ be an $\OO_K$-CM elliptic curve 
with mod $\ell$ Galois representation as in Proposition \ref{HIGHERBOTTOMPROP}a).   Then for a number field $L \supset K$, 
$\rho_{\ell}|_{\gg_L}$ has scalar image iff $\chi_{\ell} \Psi^{-2} |_{\gg_L}$ is trivial.  Since $\chi_{\ell} \colon \gg_K \ra (\Z/\ell \Z)^{\times}$ has order $\ell-1$ -- 
that is, for all $1 \leq k < \ell-1$, $\chi_{\ell}^k \neq 1$ -- and $\Psi^{-2}$ has order dividing $\frac{\ell-1}{4}$, the character $\chi_{\ell} \Psi^{-2}$ has order $\ell-1$.  Thus $[K(\PP E[\ell]):K] = \ell-1$,whereas $[K(\ell):K] = \frac{\ell-1}{2}$.  So $d(E,\ell) = 2$. \\
b) Let $K = \Q(\sqrt{-3})$, and let $\ell \equiv 7,31 \pmod{36}$.  Let $E_{/K}$ be an $\OO$-CM elliptic curve 
with mod $\ell$ Galois representation as in Proposition \ref{HIGHERBOTTOMPROP}b).  As in part a), we have
$[K(\PP E[\ell]):K] = \ell-1$ and $[K(\ell):K] = \frac{\ell-1}{3}$.  So $d(E,\ell) = 3$. 
\end{example}

\begin{prop}
\label{WEAKERPROP}
Let $\OO$ be an order of discriminant $\Delta = \ff^2 \Delta_K$, and let $N \in \Z^+$.  Then there is an $\OO$-CM elliptic curve $E_{/K(N\ff)}$ such that 
the mod $N$ Galois representation consists of scalar matrices.
\end{prop}
\begin{proof}
When $\Delta \notin \{-3,-4\}$, this is immediate from Theorem \ref{PARISHTHM}b): in that case, the elliptic curve has a 
model defined over $K(\ff)$.  Thus we may assume that $\Delta \in \{-3,-4\}$, so $\ff = 1$. Let $\zeta \in \OO_K^{\times}$ be a primitive $w_K$th root of unity.  Let $\OO(N)$ be the order in $K$ of conductor $N$, let $\tilde{E}_{/K(N)}$ be an $\OO(N)$-CM elliptic curve,  and let $\iota \colon \tilde{E} \ra E$ be the canonical $K(N)$-rational 
isogeny to an $\OO_K$-CM elliptic curve $E$, let $\iota^{\vee} \colon E \ra \tilde{E}$ be the dual isogeny, and let $C$ be the kernel of $\iota^{\vee}$.  Identifying $E[N]$ with $N^{-1} \OO_K/\OO_K \subset \C/\OO_K$, $\iota^{\vee} \colon \C/\OO_K \ra \C/\OO(N)$ is the map
$z + \OO_K \mapsto Nz + \OO(N)$, so $C$ is the $\Z$-submodule of $\C/\OO_K$ generated by $P_1 = \frac{1}{N} + \OO_K$.  Because 
$C$ is stable under the action of $\gg_{K(N)}$, this action is given by an isogeny character, say 
\[ \sigma(P_1) = \Psi(\sigma) P_1. \]
Let $P_2 = \zeta P_1$.  Then $\{P_1,P_2\}$ is a $\Z/N\Z$-basis for $E[N]$.  Moreover, for $\sigma \in \gg_{K(N)}$,
\[ \sigma P_2 = \sigma \zeta P_1 = \zeta \sigma P_1 = \zeta \Psi(\sigma) P_1 = \Psi(\sigma) \zeta P_1 = \Psi(\sigma) P_2. \]
It follows that $\sigma \in \gg_{K(N)}$ acts on $E[N]$ via the scalar matrix $\Psi(\sigma)$.  
\end{proof}

\section{Proof of Theorem \ref{STEVENHAGEN} and Its Corollaries}

\subsection{An Equality of Class Fields} Let $\OO$ and $\OO'$ be orders in an imaginary quadratic field $K$ of conductors $\ff$ and $N\ff$, respectively.  Here we prove  $K(\ff)^{N \OO_K} = K^{N \OO_K}K(N\ff)$.  We may assume that $N \geq 2$.  Class field theory gives a canonical isomorphism 
\begin{equation}
\label{CLASSICALCFT}
 \Psi \colon \Aut(K^{\ab}/K(\ff)) \stackrel{\sim}{\ra} \widehat{\OO}^{\times}/\OO^{\times} 
\end{equation}
\cite[(3.2)]{Stevenhagen01}.  Thus it suffices to prove an equality of open subgroups of $\widehat{\OO_K}^{\times}/\OO_K^{\times}$.  We abbreviate 
\[ \OO_p \coloneqq \OO \otimes \Z_p. \]
Put
\[ A \coloneqq \{x \in \widehat{\OO}^{\times} \mid x \equiv 1 \pmod{N} \} = \prod_{p \nmid N} \OO_p^{\times} 
\times \prod_{p \mid N} (1+N \OO_p), \ \tilde{A} \coloneqq A \OO_K^{\times},\]
\[ B \coloneqq \widehat{\OO'}^{\times} = \prod_p (\OO')_p^{\times}, \ \tilde{B} \coloneqq B \OO_K^{\times}, \]
\[ C \coloneqq \{x \in \widehat{\OO_K}^{\times} \mid x \equiv 1 \pmod{N} \} = \prod_{p \nmid N} (\OO_K)_p^{\times} \times
\prod_{p \mid N}  (1+N(\OO_K)_p), \  \tilde{C} \coloneqq C \OO_K^{\times}. \]
Under class field theory, the field $K(\ff)^{N \OO_K}$ corresponds to $\tilde{A}$ (cf. 
\cite[p. 9]{Stevenhagen01}), the field $K(N\ff)$ corresponds to 
$\tilde{B}$ and the field $K^{N \OO_K}$ corresponds to $\tilde{C}$, so showing that $K(\ff)^{N \OO_K} = K^{N \OO_K}K(N\ff)$ is equivalent to showing that
\[ \tilde{A} = \tilde{B} \cap \tilde{C}. \]
Step 1: We show that $A = B \cap C$.  Writing $A_p$, $B_p$ and $C_p$ for the components of $p$ of each of these groups, it is enough to show that 
\[ A_p = B_p \cap C_p \text{ for all primes } p. \]
Case 1: Suppose $p \nmid N$.  Then 
\[ A_p = \OO_p^{\times}, \]
\[ B_p = (\OO')_p^{\times} = A_p, \]
\[ C_p = (\OO_K)_p^{\times}, \]
so $C_p \supset A_p = B_p$ and thus $B_p \cap C_p = A_p$.  \\
Case 2: Suppose $p \mid N$.  Write $\OO_K = \Z 1 + \Z \tau_K$, so $\OO = \Z 1 + \Z \ff \tau_K$.  We have
\[ A_p = 1 + N\OO_p = 1 + N \Z_p 1 + N \ff \Z_p \tau_K, \]
\[ B_p = (1 + \Z_p 1 + N \ff  \Z_p \tau_K)^{\times}, \]
\[ C_p = 1 + N (\OO_K)_p = 1 + N \Z_p 1 + N \Z_p \tau_K, \]
so indeed we have $B_p \cap C_p = A_p$.  \\ \indent 
It follows that $\tilde{B} \cap \tilde{C} = B \OO_K^{\times} \cap C \OO_K^{\times} \supset A \OO_K^{\times} = \tilde{A}$, so it remains to show that $\tilde{B} \cap \tilde{C} \subset \tilde{A}$.  \\
Step 2: Suppose $\Delta_K < -4$, so $\OO_K^{\times} = \{ \pm 1\}$.  Then $\tilde{B} = B$, so if $z \in \tilde{B} \cap \tilde{C}$, 
then there is $\epsilon \in \{ \pm 1\}$ such that $z \in B$, $-z \in B$ and $\epsilon z \in C$, so $\epsilon z \in B \cap C = A$ and 
thus $z \in \tilde{A}$.  \\
Step 3: Suppose $K = \Q(\sqrt{-1})$ and let $\zeta$ be a primitive $4$th root of unity, so $\OO_K = \Z 1 + \Z \zeta$ and $\OO_K^{\times} = \{1,\zeta,\zeta^2,\zeta^3\}$ .  Suppose $z \in \tilde{B} \cap \tilde{C}$.  Then there are $i,j \in \{0,1,2,3\}$, $b \in B$ and 
$c \in C$ such that 
\[ z = \zeta^i b = \zeta^j c. \]
We have $z \in \tilde{A}$ iff $\zeta^{-j} z \in \tilde{A}$, so we may assume that $j = 0$.   If $i$ is even we may argue as in Step 2, 
so assume that $i \in \{1,3\}$, and thus we have either $\zeta b = c$ or $\zeta c = b$.  But we claim that there are no such elements 
$b$ and $c$, which will complete the argument in this case. Indeed, choose a prime $p$ dividing $N$, and let $b_p$ and $c_p$ be the components at $p$.  There is a reduction map 
\[ (\OO_K)_p \ra \OO_K \otimes \Z/p\Z = \Z/p\Z 1 + \Z/p\Z \zeta. \]
Under this map, every element of $B_p \cup C_p$ lands in $\Z/p\Z 1$, so $b_p,c_p \in \Z/p\Z 1$ while 
$\zeta b_p, \zeta c_p \in \Z/p\Z \zeta$.  Thus we cannot have $\zeta b_p = c_p$ or $\zeta c_p = b_p$.  \\ \indent
If $K = \Q(\sqrt{-3})$, then we let $\zeta$ be a primitive $6$th root of unity, so $\OO_K = \Z 1 + \Z \zeta$ and $\OO_K^{\times} = 
\{1,\zeta,\zeta^2,\zeta^3,\zeta^4,\zeta^5\}$, and the argument is very similar: we cannot have $\pm \zeta b_p = c_p$ 
or $\pm b_p = \zeta c_p$.

\subsection{Proof of Theorem \ref{MAINTHM}} By Theorems \ref{BCS3.16THM} and \ref{PARISHTHM}a) and $\S 5.1$, we have 
\[ K(\ff)(\hh(E[N])) \supset K^{N \OO_K}K(N\ff) = K(\ff)^{N \OO_K}. \]
For any $\OO$-CM elliptic curve $E_{/K(\ff)}$, the splitting field $\overline{K(\ff)}^{\Ker \overline{\rho_N}}$ of the reduced mod $N$ 
Galois representation $\overline{\rho_N}$ on $E$ (cf. \S 1.3) is $K(\ff)(\hh(E[N]))$, so 
 \[[K(\ff)(\mathfrak{h}(E[N])):K(\ff)] \leq \#\overline{C_N(\OO)}. \] 
As described in the introduction, it is immediate from (\ref{CLASSICALCFT}) and the definition of $K(\ff)^{N \OO_K}$ that 
\[ \Aut(H_{N,\OO}/K(\ff)) = \overline{C_N(\OO)},\]
and thus it follows that 
\[ K(\ff)(\hh(E[N])) = K^{N \OO_K}K(N \ff). \]

\subsection{Proof of Corollaries \ref{NEWCOR3}, \ref{COR1.2} and \ref{LargeTwistCor}}

\subsubsection{Proof of Corollary \ref{NEWCOR3}}
Theorem \ref{STEVENHAGEN} implies
that for any number field $F \supset K(j(E))$ and $N \in \Z^+$, we have 
\[ [\overline{C_N(\OO)}:\overline{\rho_N}(\gg_F)] \mid [F:K(j(E))] \]
and thus 
\[ [C_N(\OO):\rho_N(\gg_F)] \mid \# \OO^{\times} [F:K(j(E))] \leq 6 [F:K(j(E))]. \]

\subsubsection{Proof of Corollary \ref{COR1.2}} We may of course assume that $N \geq 2$.  Let $w = \# q_N(\OO^{\times})$, 
so \[w = \begin{cases} \# \OO^{\times} & N \geq 3 \\ \frac{\# \OO^{\times}}{2} & N = 2 \end{cases}.\]  
Once again we denote by $\mu_w$ the image of $\OO^{\times}$ in $C_N(\OO)$, a cyclic group of 
order $w$.   Let $E_{/K(\ff)}$ be any $\OO$-CM elliptic curve.  We may view $G = \Aut(K(\ff)(E[N])/K(\ff))$ as a subgroup of $C_N(\OO)$.  Let $H = G \cap \mu_w$ and $L =  (K(\ff)(E[N]))^H$, so a suitable twist $(E')_{/L}$ of $E_{/L}$ has trivial mod $N$ Galois representation.   As shown in the 
proof of Theorem \ref{WFP}, we have $L = K(\ff)(\hh(E[N]))$, so by Theorem \ref{MAINTHM} we have $[L:K(\ff)] = \# \overline{C_N(\OO)}$.  

\subsubsection{Proof of Corollary \ref{LargeTwistCor}}
We may assume that $N \geq 2$.  Let $q_N \colon \OO^{\times} \ra C_N(\OO)$ be the natural homomorphism.  By Theorem \ref{WFP}b), there is an elliptic curve $E_{/K(\ff)}$ 
such that \[[K(\ff)(E[N]):K(\ff)(\hh(E[N]))] = \# q_N(\OO^{\times}) = \begin{cases} \# \OO^{\times} & N \geq 3 \\
\frac{\# \OO^{\times}}{2} & N = 2 \end{cases}. \] By Theorem \ref{MAINTHM} we have $[K(\ff)(\hh(E[N])):K(\ff)] = \# \overline{C_N(\OO)}$.  Thus $\rho_{E,N}(\gg_{K(\ff)}) = C_N(\OO)$.  

\section{Applications}

\subsection{Divisibility in Silverberg's Theorem}

\begin{lemma}
\label{GTLEMMA}
Let $J,M$ be subgroups of a group $G$.  If $M$ is normal and $J \cap M = \{1\}$, then 
 $\# J \mid [G:M]$.
\end{lemma}
\begin{proof}
The composite homomorphism $J \hookrightarrow G \ra G/M$ is an injection.
\end{proof}
\noindent
The following result extends \cite[Cor. 2.5]{BCP} from maximal orders to all imaginary quadratic orders, thereby confirming 
the expectation expressed in \cite[Remarks 2.2]{BCP}.

\begin{thm}
\label{SPY}
Let $\OO$ be an order in an imaginary quadratic field $K$, and let $E$ be an $\OO$-CM elliptic curve defined over a number field $F \supset K$.  
If $E(F)$ has a point of order $N \in \Z^+$, then 
\[  \varphi(N) \mid \frac{ \# \OO^{\times}}{2} \frac{[F:\Q]}{\# \Pic \OO}. \]
\end{thm}
\begin{proof}
Let $\mathcal{I}_N = [C_N(\OO):\rho_N(\gg_F)]$ be the index of the mod $N$ Galois representation in the Cartan subgroup.  By Corollary \ref{COR1.4} we have
\[ \mathcal{I}_N \mid \# \OO^{\times} [F:K(j(E))] = 
\frac{\# \OO^{\times}}{2} \frac{ [F:\Q]}{ \# \Pic \OO}. \]  Since there is a rational point of order $N$, the subgroup $\rho_N(\gg_F)$ contains no scalar
 matrices other than the identity.  Appying Lemma \ref{GTLEMMA} with $G = C_N(\OO)$, $M = \rho_N(\gg_F)$ and $J$ the subgroup of 
scalar matrices, we get $\varphi(N) \mid  \mathcal{I}_N$, and we are done.
\end{proof}

\subsection{A Theorem of Franz}
Let $\OO$ be an order in $K$, of conductor $\ff$, and let $E_{/K(\ff)}$ be an $\OO$-CM elliptic curve.  Choose a 
field embedding $K(\ff) \hookrightarrow \C$ such that $j(E) = j(\C/\OO)$ and an isomorphism $E(\C) \stackrel{\sim}{\ra} \C/\OO$.  This induces an isomorphism $E(\overline{K(\ff)})[\tors] \stackrel{\sim}{\ra} \C/\OO[\tors]$, which we use to view (the image in $\C/\OO$ of) 
$\tau_K = \frac{\Delta_K + \sqrt{\Delta_K}}{2}$  as a point of $E(\overline{K(\ff)})[\tors]$ of order $\ff$.

 \begin{thm}(Franz \cite{Franz35})
With notation as above, we have 
\[ K(\ff)(\hh(\tau_K)) = K^{(\ff)}. \]
 \end{thm}
 
 \begin{proof}
As in the proof of Theorem \ref{ITT}, over $\C$ we may view the canonical isogeny as $\iota \colon \mathbb{C}/\O \rightarrow \mathbb{C}/\Ok$. We take $e_1=\frac{1}{\ff}+ \O$ and $e_2=\tau_K+\O$ as a basis for $E[\ff]$. Then $e_2$ generates $\ker(\iota)$, a $K(\ff)$-rational cyclic subgroup of order $\ff$, and with respect to $\{e_1,e_2\}$ the image of the mod $\ff$ Galois representation associated to $E_{/K(\ff)}$ consists of matrices of 
the form 
\[
\left[ \begin{array}{cc} a & b \ff^2 \frac{\Delta_K-\Delta_K^2}{4} \\ b & a + b \ff \Delta_K \end{array} \right] \text{with } a,b \in \Z/\ff\Z.
\]
Viewing entries mod $\ff$, we see there is a character $\Psi \colon \mathfrak{g}_{K(\ff)} \rightarrow (\Z/\ff\Z)^{\times}$ such that  
\[  \rho_{E,\ff}(\sigma) =\left[ \begin{array}{cc} \Psi(\sigma) & 0 \\ * & \Psi(\sigma)  \end{array} \right]. \] 
If $\ff \leq 2$, then $K(\ff)(\mathfrak{h}(\tau_K)) = K(\ff) = K^{(\ff)}$ and the result holds.   Thus we may assume $\ff \geq 3$. Let $L\coloneqq K(\ff)(\hh(e_2))$. Since $\ff \geq 3$, we have $j(E) \neq 0, 1728$, so $[L(e_2):L]$ divides 2 and the restriction $\Psi |_{\mathfrak{g}_L} \colon \mathfrak{g}_L \rightarrow \{\pm 1\}$ defines a quadratic character $\chi$. On the twist $E^{\chi}$ of $E_{/L}$ the point $e_2$ becomes $L$-rational. As in the proof of Theorem 5.5 of \cite{BCS},  let $\Psi^{\pm} \colon \mathfrak{g}_{K(\ff)} \rightarrow (\Z/\ff\Z^{\times})/\{\pm 1\}$ denote the composition of $\Psi$ with the natural map $(\Z/\ff\Z)^{\times} \ra (\Z/\ff\Z)^{\times}/{\pm 1}$. Then $L \subset (\overline{K(\ff)})^{\ker \Psi^{\pm}}$, so $[L:K(\ff)] \mid \frac{\varphi(\ff)}{2}$.  If $\iota \colon E^{\chi} \rightarrow E'$ is the canonical isogeny, then the proof of Theorem \ref{ITT} shows that $\iota(e_1)$ is an element of $E'(L)$ which generates $E'[\ff]$ as an $\Ok$-module. Thus $E'$ has full $\ff$-torsion over $L$, so by Theorem \ref{FIRSTCM}, $K^{(\ff)} \subset L$.  So
\[ [L:K(\ff)] \geq [K^{(\ff)}:K(\ff)] = \frac{ \varphi(\ff)}{2} \geq [L:K(\ff)], \]
and thus $K(\ff)(\hh(e_2)) = L =  K^{(\ff)}$.
 \end{proof}
 
 \subsection{The Field of Moduli of a Point of Prime Order}
Let $K \neq \Q(\sqrt{-1}), \Q(\sqrt{-3})$ be an imaginary quadratic field, and let $\OO \subset K$ be the order of conductor $\ff$. Here we use Theorem \ref{ITT} to determine the smallest field $F \supset K$ for which there exists an $\OO$-CM elliptic curve $E_{/F}$ with an $F$-rational point of order $\ell>2$.
\begin{lemma}
\label{INTLEM}
Let $K$ be an imaginary quadratic field, let $\ff \in \Z^+$, and let $\ell>2$ be prime. Then $K^{(\ell)} \cap K(\ell \ff)=K(\ell)$.
\end{lemma}

\begin{proof}
Let $\Delta = \ff^2 \Delta_K$.  The statement is immediate if $\ff=1$, so suppose $\ff>1$. By Theorem \ref{COX7.28THM}, 
\[[K(\ell \ff):K(\ff)]= \ell - \left(\frac{\Delta}{\ell}\right). \]

Since $[K^{(\ell)}K(\ell \ff):K(\ff)]=\#C_{\ell}(\OO)/2$ by Theorem \ref{MAINTHM}, we have in both cases that 
\[
[K^{(\ell)}K(\ell \ff):K(\ell \ff)]=\frac{\#C_{\ell}(\OO)}{2[K(\ell \ff):K(\ff)]}=\frac{1}{2}(\ell-1).
\]
Thus $[K^{(\ell)}:K^{(\ell)} \cap K(\ell\ff)]=[K^{(\ell)}K(\ell \ff):K(\ell \ff)]=\frac{1}{2}(\ell-1)$. As we have $K(\ell) \subset K^{(\ell)} \cap K(\ell\ff)$ and $[K^{(\ell)}:K(\ell)]=\frac{1}{2}(\ell-1)$, the result follows.
\end{proof}

\begin{thm}
\label{PRIMEMODULITHM}
 Let $K \neq \Q(\sqrt{-1}), \Q(\sqrt{-3})$ be an imaginary quadratic field, and let $\OO$ be the order of conductor $\ff$ in $K$.
Let $F \supset K$. \\
a) Let $E_{/F}$ be an $\OO$-CM elliptic curve such that $E(F)$ contains a point of prime order $\ell>2$.  Then there is a prime $\pp$ of $\OO_K$ lying over $\ell$ such that $K(\ff)K^{\mathfrak{p}} \subset F$. \\
b) If $\leg{\Delta}{\ell} \neq -1$, then there is a prime $\pp$ of $\OO_K$ lying over $\ell$ and an $\OO$-CM elliptic curve $E_{/K(\ff)K^{\mathfrak{p}} }$ such that $E(K(\ff)K^{\mathfrak{p}})$ has a point of order $\ell$.
\end{thm}

\noindent If $\leg{\Delta}{\ell} = -1$, then an $\OO$-CM elliptic curve $E_{/F}$ with an $F$-rational point of order $\ell$ must have full $\ell$-torsion (see \cite[Thm. 4.8]{BCS} or Lemma \ref{NEWLEMMA2}). In this case, $K(\ell\ff)K^{(\ell)}\subset F$ by Theorem \ref{MAINTHM}. The existence of an elliptic curve $E_{/K(\ell\ff)K^{(\ell)}}$ with full $\ell$-torsion is guaranteed by Corollary 
\ref{1.1C}.

\begin{proof}
a) Let $F \supset K$ and $E_{/F}$ be an $\OO$-CM elliptic curve with an $F$-rational point of order $\ell$.  By Theorem \ref{ITT}, there is an $\Ok$-CM elliptic curve $E'_{/F}$ with an $F$-rational point $P$ of order $\ell$. If $M$ is the $\Ok$-submodule of $E'(F)$ generated by $P$, then $M=E'[\ann M]$ and $\#M=|\! \ann M|$ by Theorem \ref{Thm2.9}. Since $\ell \mid \# M$, we must have $\mathfrak{p} \mid \ann M$ for some prime $\mathfrak{p}$ of $\Ok$ above $\ell$. By Theorem \ref{FIRSTCM} we have 
\[K(\ff) K^{\mathfrak{p}} \subset K(\ff) K^{\ann M}=K(j(E)) K^{(1)}(\mathfrak{h}(E'[\ann M])) \subset F. \]
b) If $\leg{\Delta}{\ell} \neq -1$, then an $\OO$-CM elliptic curve $E_{/K(\ff)}$ possesses a $K(\ff)$-rational cyclic subgroup of order $\ell$. (See e.g. \cite[p.13]{TORS1}.  This is also a special case of Theorem \ref{KJISOGTHM}.) By \cite[Thm. 5.5]{BCS}, there is an extension $L/K(\ff)$ of degree $(\ell-1)/2$ and a quadratic twist $(E_1)_{/L}$ such that $E_1(L)$ has a point of order $\ell$.  By 
part a), there is a prime $\pp$ of $\OO_K$ lying over $\ell$ such that $K(\ff)K^{\mathfrak{p}} \subset L$, so it will suffice to show that $[K(\ff)K^{\mathfrak{p}}:K(\ff)] \geq \frac{\ell-1}{2}$.
\\ \indent 
If $\ell \nmid \ff$, then primes above $\ell$ are unramified in $K(\ff)/K^{(1)}$. Thus $K(\ff)$, $K^{\mathfrak{p}}$ are linearly disjoint over $K^{(1)}$, and we have  $[K(\ff)K^{\mathfrak{p}}:K(\ff)]=[K^{\mathfrak{p}}:K^{(1)}]=\frac{1}{2}(\ell-1)$ since $\leg{\Delta_K}{\ell}= \leg{\Delta}{\ell} \neq -1$. If $\ell \mid \ff$, then applying Lemma \ref{INTLEM} with $\frac{\ff}{\ell}$ in place of $\ff$, we have \[ K^{\mathfrak{p}} \cap K(\ff) \subset K^{(\ell)} \cap K(\ff)
=K(\ell). \]  
Thus $K^{\pp} \cap K(\ff)  = K^{\pp} \cap K(\ell)$, so 
\[ [K(\ff)K^{\pp}:K(\ff)]=[K^{\pp}:K^{\pp} \cap K(\ff)]=[K^{\pp}:K^{\pp} \cap K(\ell)] = [K(\ell)K^{\pp}:K(\ell)] \]
and it is enough to show that $[K(\ell)K^{\pp}:K(\ell)] \geq \frac{\ell-1}{2}$.  
\begin{itemize}
\item $\leg{\Delta_K}{\ell} = 1$: We will prove that $K^{\pp} \cap K(\ell) =K^{(1)}$ using CM elliptic curves. Let $(E_0)_{/K^{(1)}}$ be an $\Ok$-CM elliptic curve. Then $E_0[\mathfrak{p}]$ is stable under the action of $\gg_{K^{(1)}}$ and generated by a point $P$ of order $\ell$. By \cite[Thm. 5.5]{BCS}, there is an extension $L/K^{(1)}$ of degree $(\ell-1)/2$ and a quadratic twist $(E_1)_{/L}$ such that $P$ becomes $L$-rational. By Theorem \ref{FIRSTCM} we have $K^{\mathfrak{p}} \subset L$, and $K^{\mathfrak{p}} = L$ since $[K^{\mathfrak{p}}:K^{(1)}]=\frac{1}{2}(\ell-1)$.  Over $K(\ell)K^{\mathfrak{p}}$, the curve $E_1$ has a rational point of order $\ell$, and the mod $\ell$ Galois representation is scalar by Theorem \ref{PARISHTHM}. Thus $E_1$ has full $\ell$-torsion over $K(\ell)K^{\mathfrak{p}}$, and $K^{(\ell)} \subset K(\ell)K^{\mathfrak{p}}$. This implies $\frac{1}{2} (\ell-1) \mid [K(\ell)K^{\mathfrak{p}}:K(\ell)]=[K^{\mathfrak{p}}:K^{\pp} \cap K(\ell) ]$. Since $[K^{\mathfrak{p}}:K^{(1)}]= \frac{1}{2} (\ell-1)$, we have $K^{\pp} \cap K(\ell) =K^{(1)}$, and $[K(\ff)K^{\mathfrak{p}}:K(\ff)]=[K^{\mathfrak{p}}:K^{(1)}]=\frac{1}{2}(\ell-1)$.
\item $\leg{\Delta_K}{\ell} = -1$: In this case, $K^{\mathfrak{p}}=K^{(\ell)}$, so $K^{\pp} \cap K(\ell) =K(\ell)$. This implies $[K(\ff)K^{\mathfrak{p}}:K(\ff)]=[K^{\mathfrak{p}}:K(\ell)]=\frac{1}{2}(\ell-1)$.
\item $\leg{\Delta_K}{\ell} = 0$: Since $[K(\ell):K^{(1)}]=\ell$ and $[K^{\mathfrak{p}}:K^{(1)}]=\frac{1}{2}(\ell-1)$, we have $K^{\pp} \cap K(\ell) =K^{(1)}$. Thus $[K(\ff)K^{\mathfrak{p}}:K(\ff)]=[K^{\mathfrak{p}}:K^{(1)}]=\frac{1}{2}(\ell-1)$.  \qedhere
\end{itemize}
\end{proof}

\begin{remark}
Assume the setup of Theorem \ref{PRIMEMODULITHM} but take $K = \Q(\sqrt{-1})$ or $K = \Q(\sqrt{-3})$.  Then the assertion of 
Theorem \ref{PRIMEMODULITHM}b) is false.  Indeed, if $\ell \geq 5$ and $\leg{\Delta}{\ell} \neq -1$, we have $[K(\ff)K^{\mathfrak{p}}:K(\ff)] \mid \frac{1}{w_K}(\ell-1)$. (See Lemma \ref{CFTLEMMA}.) Suppose $F \supset K$, and let $E_{/F}$ be an elliptic curve with CM by the order in $K$ of conductor $\ff$. If $E(F)$ contains a rational point of order $\ell$, then Theorem \ref{SPY} implies $\frac{1}{2} (\ell-1) \mid [F:K(\ff)]$. Thus $F$ must properly contain $K(\ff)K^{\mathfrak{p}}$.
\end{remark}

\subsection{Sharpness in the Isogeny Torsion Theorem}
\textbf{} \\ \\ \noindent
The following result was established during the proof of Theorem \ref{ITT}.

\begin{lemma}
\label{XXLEMMA1}
Let $E$ be an $\OO$-CM elliptic curve defined over a number field $F$ containing the CM field $K$, and for a positive integer 
$\ff'$ dividing the conductor $\ff$ of $\OO$, let $\iota \colon E \ra E'$ 
be the canonical $F$-rational isogeny to an elliptic curve $E'$ with CM by the order in $K$ of conductor $\ff'$.  Write 
\[ E(F)[\tors] = \Z/s\Z \times \Z/e\Z, \ E'(F)[\tors] = \Z/s'\Z \times \Z/e'\Z, \]
where $s \mid e$ and $s' \mid e'$. Then $s \mid s'$.
\end{lemma}

\noindent
In \cite[$\S$4]{Ross94}, 
Ross claims that if $E$ is a CM elliptic curve defined over a number field $F$ containing the CM field, then the exponent of 
the finite group $E(F)[\tors]$ is an invariant of the $F$-rational isogeny class.  In the setting of Lemma \ref{XXLEMMA1}, 
this would give $e = e'$, and combining this with the conclusion of Lemma \ref{XXLEMMA1} we would get an injective group homomorphism $E(F)[\tors] \hookrightarrow E(F')[\tors]$.  This conclusion is stronger than that of 
Theorem \ref{TIT}.  However Ross's claim is false: in the setup of Lemma \ref{XXLEMMA1}, one can have $e' < e$ (in which 
case there is no injective group homomorphism $E(F)[\tors] \hookrightarrow E'(F)[\tors]$), as the following result shows.  

\begin{prop}
\label{ABBEYPROP3}
Let $\ell > 3$ be a prime number, let $K = \Q(\sqrt{-\ell})$,  let $n \in \Z^{\geq 3}$, let $\OO$ be 
the order in $K$ of conductor $\ff=\ell^{ \lfloor \frac{n}{2} \rfloor}$, and let $F = K(\ff)$.  For any $\OO$-CM elliptic curve $E_{/F}$, there is an extension $L/F$ of degree $\varphi(\ell^n)$ such that $E(L)$ has a point of order $\ell^n$, and no $\OO_K$-CM elliptic curve has an $L$-rational point of order $\ell^k$ for $k>\frac{1}{2} \left(n+1 + \lfloor \frac{n}{2} \rfloor \right)$ (hence 
no $L$-rational point of order $\ell^n$).
\end{prop}

\begin{proof}
Let $E_{/F}$ be an $\OO$-CM elliptic curve. As in (\ref{ITTEQ1}) we may choose a basis $\{e_1,e_2\}$ for $E[\ell^n]$ so that the image of the mod $\ell^n$ Galois representation consists of matrices 
\[  \left[ \begin{array}{cc} a & b \ff^2 \frac{\Delta_K-\Delta_K^2}{4} \\ b & a+b\ff\Delta_K  \end{array} \right] \text{ with } a,b \in \Z/\ell^n\Z . \]
Since $\ell$ ramifies in $K$ and $\ff = \ell^{\lfloor \frac{n}{2} \rfloor}$, we have $\ord_{\ell}(b \ff^2 \frac{\Delta_K-\Delta_K^2}{4}) = 
1 + 2 \lfloor \frac{n}{2} \rfloor \geq n$, so the matrices have the form  
\[ \left[ \begin{array}{cc} a & 0 \\ b & a + b\ff \Delta_K \end{array} \right]  \text{ with }  a,b \in \Z/\ell^n \Z. \]

The action of $\mathfrak{g}_F$ on $\langle e_2 \rangle$ gives a character $\Phi \colon \mathfrak{g}_F \rightarrow (\Z/\ell^n\Z)^{\times}$. Take $M=(\overline{F})^{\ker \Phi}$. Then $[M:F] \mid \varphi(\ell^n)$ and $\Phi |_{\mathfrak{g}_M}$ is trivial. Thus there exists an extension $L/F$ with $[L:F]= \varphi(\ell^n)$ such that $E(L)$ contains $e_2$.

Let $E'_{/L}$ be an $\OO_K$-CM elliptic curve, and suppose $E'(L)$ contains a point $P$ of order $\ell^k$.  Let $\pp$ be the prime ideal of $\OO_K$ such that $\ell\Ok=\mathfrak{p}^2$.  We claim that the $\OO_K$-submodule 
$M = \langle P \rangle_{\OO_K}$ of $E'(L)$ generated by $P$ contains $E[\pp^{2k-1}]$ and thus, by Theorem \ref{FIRSTCM}, that $K^{\mathfrak{p}^{2k-1}} \subset L$. 
Indeed, by Theorem \ref{Thm2.9}, we have $M = E[I]$ for some ideal $I$ of $\OO_K$ such that $(\OO_K/I,+)$ has $\ell$-power 
order and exponent $\ell^k$.  Since $\ell$ ramifies in $\OO_K$, this forces $I$ to be of the form $\pp^a$ for some $a \in \Z^+$, 
and the smallest $a$ such that $(\OO_K/\pp^a,+)$ has exponent $\ell^k$ is $a = 2k-1$, establishing the claim. Thus 
\[ \ord_{\ell}([K^{\pp^{2k-1}}:K^{(1)}]) =2k-2 \leq  \ord_{\ell}([L:K^{(1)}])= \left\lfloor \frac{n}{2} \right\rfloor +n-1,\]
so $k \leq \frac{1}{2}(n+1+ \lfloor \frac{n}{2} \rfloor)$. \qedhere 
\end{proof}
\noindent
In the setting of Theorem \ref{ITT}, one wonders whether $\# E(F)[\tors] = \# E'(F)[\tors]$.  In fact $\frac{ \# E'(F)[\tors]}{\#E(F)[\tors]}$ can be arbitrarily large:

\begin{prop}
\label{ABBEYPROP2}
Let $\ell$ be an odd prime, let $K \neq \Q(\sqrt{-1}), \Q(\sqrt{-3})$ be an imaginary quadratic field, let 
$\OO$ be the order in $K$ of conductor $\ell$, and let $F = K(\ell)$.  For any $\OO$-CM elliptic curve 
$E_{/F}$ there is an extension $L/F$ such that if $\iota \colon E \ra E'$ is the canonical isogeny to an $\OO_K$-CM elliptic curve $E'$, then 
\[ \ell \mid \frac{ \# E'(L)[\tors]}{\# E(L)[\tors]}. \]
\end{prop}

\begin{proof}
Let $E_{/F}$ be an $\OO$-CM elliptic curve. As above, there is a basis $\{e_1,e_2\}$ for $E[\ell]$ such that
\[ \rho_{\ell}(\mathfrak{g}_F) \subset\left\{ \left[ \begin{array}{cc} a & 0 \\ b & a \end{array} \right] \mid a,b \in \Z/\ell\Z \right\} \]
and there is an extension $L/F$ with $[L:F]=\ell-1$ such that $E(L)$ contains $e_2$. In fact, $E(L)[\ell^{\infty}] \cong \Z/\ell\Z$. Indeed, $E$ does not have full $\ell$-torsion over $L$ since Theorem \ref{MAINTHM} would imply $K^{(\ell)}K(\ell^2) \subset L$ and $\frac{1}{2} \ell(\ell-1) = [K^{(\ell)}K(\ell^2):K(\ell)]$. In addition, $E(L)$ has no point of order $\ell^2$ by Theorem \ref{SPY}.

Let $\iota \colon E \rightarrow E'$ be the canonical $L$-rational isogeny from $E_{/L}$ to $E'_{/L}$, where $E'$ has $\OO_K$-CM. Since $e_2 \in E(L)$, the second paragraph of $\S3$ shows that $\iota(e_1) \in E'(L)$, and $\iota(e_1)$ generates $E'[\ell]$ as an $\Ok$-module. In other words, $\Z/\ell\Z \times \Z/\ell\Z \hookrightarrow E'(L)[\tors]$. It follows that $\ell \mid \frac{ \# E'(L)[\tors]}{\# E(L)[\tors]}$.
\end{proof}

\noindent

Finally, Theorem \ref{ITT} requires $K \subset F$.  This hypothesis cannot be omitted:

\begin{prop}
\label{ABBEYPROP1}
Let $\ell > 3$ be a prime with $\ell \equiv 3 \pmod{4}$ and let $n \in \Z^{\geq 3}$.  Let $K = \Q(\sqrt{-\ell})$, and 
let $\OO$ be the order in $K$ of conductor $\ff = \ell^{\lfloor \frac{n}{2} \rfloor}$.  Let $F = \Q(j(\C/\OO))$.  There 
is an elliptic curve $E_{/F}$ and an extension $L/F$ of degree $\frac{\varphi(\ell^n)}{2}$ such that: \\
(i) $L \not \supset K$, \\
(ii) $E(L)$ has a point of order $\ell^n$, and \\
(iii) for every $\OO_K$-CM elliptic curve $E'_{/L}$ we have $\ell^n \nmid \# E'(L)[\tors]$.
\end{prop}

\begin{proof}
Let $E_{/F}$ be an $\OO$-CM elliptic curve. By \cite[Corollary 4.2]{Kwon99}, $E$ has an $F$-rational subgroup which is cyclic of order $\ell^n$. It follows from \cite[Theorem 5.6]{BCS} that there is a twist $E_1$ of $E_{/F}$ and an extension $L/F$ of degree $\varphi(\ell^n)/2$ such that $E_1(L)$ has a point of order $\ell^n$. Note $[L:\Q]=h_K\ell^{ \left \lfloor{\frac{n}{2}}\right \rfloor}\frac{\varphi(\ell^n)}{2}$ is odd (see \cite[Proposition 3.11]{Cox89}) , so $K \not\subset L$.

Let $E'_{/L}$ be an $\OO_K$-CM elliptic curve. Since $[L:\Q]$ is odd, $E'(L)[\ell^{\infty}]$ must be cyclic, as full $\ell^k$-torsion would imply $\Q(\zeta_{\ell^k}) \subset L$ by the Weil pairing. As in the last paragraph of the proof of Proposition \ref{ABBEYPROP3}, $E'(LK)$ contains no point of order $\ell^n$.  Hence $E'(L)$ contains no point of order $\ell^n$, and $\ell^n \nmid \#E'(L)[\tors]$.
\end{proof}

\subsection{Minimal and Maximal Cartan Orbits}
Let $\OO$ be an order, let $N \in \Z^+$, and let $P \in \OO/N\O$ be a point of order $N$. Since $C_N(\OO)$ contains all scalar 
matrices, if $P \in \OO/N\OO$ has order $N$, then the orbit of $C_N(\OO)$ on $P$ has size at least $\varphi(N)$.  On the other 
hand, the orbit of $C_N(\OO)$ on $P$ is certainly no larger than the number of order $N$ points of $\OO/N\OO$. 
\\ \\
In this 
section we will find all pairs $(\OO,N)$ for which there exists a Cartan orbit of this smallest possible size and also 
all pairs for which there exists a Cartan orbit of this largest possible size.
\\ \\
We introduce the shorthand $H(\OO,N)$ to mean: \emph{there is a point $P$ of order $N$ in $\OO/N\OO$ such that the $C_N(\OO)$-orbit of $P$ has size $\varphi(N)$.}

\begin{lemma}
\label{NEWLEMMA1}
Let $\OO$ be an order, and let $N = \ell_1^{a_1} \cdots \ell_r^{a_r} \in \Z^+$.  Then $H(\OO,N)$ holds 
iff $H(\OO,\ell_i^{a_i})$ holds for all $1 \leq i \leq r$.
\end{lemma}
\begin{proof}
This is an easy consequence of the Chinese remainder theorem.
\end{proof}

\begin{lemma}
\label{NEWLEMMA2}
Let $\OO$ be the order of discriminant $\Delta$, $\ell$ a prime number and $a \in \Z^+$.  \\
a) If $\left(\frac{\Delta}{\ell} \right) = 1$,  there is an $\OO$-submodule of $\OO/\ell^a \OO$ with underyling 
$\Z$-module $\Z/\ell^a \Z$.  \\
b) If $\left(\frac{\Delta}{\ell} \right) = -1$, then $C_{\ell^a}(\OO)$ acts simply transitively on the order $\ell^a$ 
elements of $\OO/\ell^a \OO$. 
\end{lemma}
\begin{proof}
a) if $\left( \frac{\Delta}{\ell} \right) = 1$, then $\OO/\ell \OO = \OO_K/\ell \OO_K \cong \Z/\ell \Z \times \Z/\ell \Z$, so 
$\OO \otimes \Z_{\ell}$ is isomorphic as a ring to $\Z_{\ell} \times \Z_{\ell}$ (see e.g. 
\cite[Cor. 7.5]{Eisenbud}) and thus $\OO/\ell^a \OO$ is isomorphic as a ring to $\Z/\ell^a \Z \times \Z/\ell^a \Z$.  \\
b) If $\left( \frac{\Delta}{\ell} \right) = -1$, then $\OO \otimes \Z_{\ell} = \OO_K \otimes \Z_{\ell}$ is a complete 
DVR with uniformizer $\ell$, so the ring $\OO/\ell^a \OO$ is finite, local and principal with maximal ideal $\langle \ell \rangle$.  
An element of $\OO/\ell^a \OO$ has order $\ell^a$ iff it lies in the unit group $C_{\ell^a}(\OO)$.
\end{proof}

\begin{lemma}
\label{NEWLEMMA3}
Let $\OO$ be the order of discriminant $\Delta$, and let $N \in \Z^+$.  The following are equivalent: \\
(i) If $2 \mid N$, then $\left( \frac{\Delta}{2} \right) \neq 1$. \\
(ii) The $\Z/N\Z$-subalgebra of $\OO/N\OO$ generated by $C_N(\OO)$ is 
 $\OO/N\OO$. 
\end{lemma}
\begin{proof}
Using the Chinese remainder theorem we reduce to the case of $N = \ell^a$ a power of 
a prime number $\ell$.  Let $B$ be the $\Z/\ell^a \Z$-subalgebra generated by $C_{\ell^a}(\OO)$, so $\# B = \ell^b$ for 
some $b \leq 2a$. \\
(i) $\implies$ (ii): Since $0 \in B \setminus C_{\ell^a}(\OO)$, we have 
 \[\# B\geq \# C_{\ell^a}(\OO) +1 \] \[= \ell^{2a} \left(1-\frac{1}{\ell}\right)\left(1-\left(\frac{\Delta}{\ell} \right) \frac{1}{\ell}\right) +1  \geq 
\begin{cases}
\frac{4}{9} \ell^{2a} + 1 > \ell^{2a-1}, & \text{ if }  \ell \geq 3 \\
\frac{1}{2} \ell^{2a} + 1 > \ell^{2a-1}, & \text{ if } \ell =2 \text{ and } \left(\frac{\Delta}{2} \right) \neq  1 \end{cases}. \]
Thus $b = 2a$ and $B = \OO/\ell^a \OO$.  \\
$\neg$ (i) $\implies$ $\neg$ (ii): If $\ell= 2$ and $\left( \frac{\Delta}{2} \right) = 1$, then 
\[\OO/2^a \OO \cong \left\{ \left[ \begin{array}{cc} \alpha & 0 \\ 0 & \beta \end{array} \right] \mid \alpha,\beta \in \Z/2^a \Z \right\} \]
and $C_{2^a}(\OO)$ consists of the set of such matrices with $\alpha,\beta \in (\Z/2^a \Z)^{\times}$.  Thus $C_{2^a}(\OO)$ 
is contained in the subalgebra 
\[ \mathcal{B} =  \left\{ \left[ \begin{array}{cc} \alpha & 0 \\ 0 & \beta \end{array} \right] \mid \alpha,\beta \in \Z/2^a \Z \text{ and } \alpha \equiv \beta \pmod{2} \right\} \]
of order $2^{2a-1}$, so $B \subset \mathcal{B} \subsetneq \OO/2^a \OO$.\footnote{Since $\# B \geq \# C_{2^a}(\OO) + 1 = 2^{2a-2}+1 > 
2^{2a-2}$, in fact we have $B = \mathcal{B}$.}
 \end{proof}

\begin{lemma}
\label{NEWLEMMA4}
For an order $\OO$ and $N \in \Z^+$, the following are equivalent: \\
(i) There is an ideal $I$ of $\OO$ with $\OO/I \cong \Z/N\Z$.  \\
(ii) There is an $\OO$-submodule of $\OO/N\OO$ with underlying commutative group $\Z/N\Z$.  \\
(iii) $H(\OO,N)$ holds.
 \end{lemma}
\begin{proof}
(i) $\iff$ (ii): \\
Step 1: Let $\Lambda$ be a free, rank $2$ $\Z$-module, and let $\Lambda'$ be a $\Z$-submodule of $\Lambda$ containing $N\Lambda$.  
By the structure theory of modules over a PID, there is a $\Z$-basis $e_1,e_2$ for $\Lambda$ and positive integers $a \mid b$ 
such that $a e_1, b e_2$ is a $\Z$-basis for $\Lambda'$.  Thus 
\[ \Lambda/\Lambda' \cong \Z/a\Z \oplus \Z/b\Z,  \ \Lambda'/N\Lambda \cong \Z/(N/b) \Z \oplus \Z/(N/a)\Z. \]
It follows that $\Lambda/\Lambda' \cong \Z/N\Z \iff \Lambda'/N\Lambda \cong \Z/N\Z$.  \\
Step 2: If $I$ is an ideal of $\OO$ with $\OO/I \cong \Z/N\Z$, then $I \supset N\OO$, so $I/N\OO \cong \Z/N\Z$ by Step 1.  
Let $M$ be an $\OO$-submodule of $\OO/N\OO$ with underyling $\Z$-module $\Z/N\Z$.  
Then $M = I/N\OO$ for an ideal $I$ of $\OO$, and by Step 1 we have $\OO/I \cong \Z/N\Z$.  \\
(ii) $\implies$ (iii): Let $P \in \OO/N\OO$ have order $N$ such that the subgroup generated by $P$ is an 
$\OO$-submodule.   For all $g \in C_N(\OO)$, $gP = a_g P$ for $a_g \in (\Z/N\Z)^{\times}$.  Conversely, 
since $C_N(\OO)$ contains all scalar matrices, the orbit of $C_N(\OO)$ on $P$ has size $\varphi(N)$. \\
(iii) $\implies$ (ii): Case 1: Suppose $2 \nmid N$ or $\left(\frac{\Delta}{2} \right) \neq 1$.  Let $P \in \OO/N\OO$ be a point of order $N$ with $C_N(\OO)$-orbit of size $\varphi(N)$.  There is a $\Z/N\Z$-basis  $e_1,e_2$ of $\OO/N\OO$ with $e_1 = P$, and our hypothesis gives that 
with respect to this basis $C_N(\OO)$ lies in the subalgebra $\left\{\left[ \begin{array}{cc} a & b \\ 0 & d \end{array} \right] \mid a,b,d \in \Z/N\Z \right\}$ of upper triangular matrices.  By Lemma \ref{NEWLEMMA3}, $\OO/N\OO$ also lies in the subalgebra 
of upper triangular matrices, and thus $\langle P \rangle$ is an $\OO$-stable submodule with underlying $\Z$-module $\Z/N\Z$. \\ \indent
Case 2: Suppose $2 \mid N$ and $\left(\frac{\Delta}{2} \right) = 1$, and write $N = 2^a N'$ with $2 \nmid N'$.  By Lemma \ref{NEWLEMMA2} and the equivalence of (i) and (ii), there is an ideal $I_1$ in $\OO$ with $\OO/I_1 \cong \Z/2^a \Z$, and by 
Case 1 there is an ideal $I_2$ in $\OO$ with $\OO/I_2 \cong \Z/N' \Z$.  By the Chinese remainder theorem, we have $\OO/I_1 I_2 \cong \Z/N\Z$.  Since (i) $\iff$ (ii), this suffices.
\end{proof}

\begin{thm}
\label{SQUAREDISCTHM}
Let $\OO$ be an order of discriminant $\Delta$, and let $N \in \Z^+$.  The following are equivalent: \\
(i) $H(\OO,N)$ holds. \\
(ii) $\Delta$ is a square in $\Z/4N\Z$.
\end{thm}
\begin{proof}
Using Lemma \ref{NEWLEMMA1}, we reduce to the case in which $N = \ell^a$ is a 
power of a prime number $\ell$. \\ 
Case 1 ($\ell$ is odd): Since $\gcd(4,\ell^a) = 1$, we may put $D = \frac{\Delta}{4} \in \Z/\ell^a \Z$.  Then $\Delta$ is a square in $\Z/4\ell^a \Z$ iff $D$ is a square in $\Z/\ell^a \Z$, and 
\begin{equation}
\label{QUADORDEREQ}
 \OO/\ell^a \OO \cong (\Z/\ell^a \Z)[t]/(t^2-D). 
\end{equation} 
If there is $s \in \Z/\ell^a \Z$ such that $D = s^2$, then 
\[ \OO/\ell^a \OO \cong (\Z/\ell^a \Z)[t]/((t+s)(t-s)), \]
so if $I$ is the ideal $\langle t+s, \ell^a \rangle$ of $\OO$, then $\OO/I \cong \Z/\ell^a \Z$.  By Lemma \ref{NEWLEMMA4}, $H(\OO,\ell^a)$ holds.  Conversely, suppose $H(\OO,\ell^a)$ holds, so by Lemma \ref{NEWLEMMA4} there is an ideal $I$ of $\OO$ with $\OO/I \cong \Z/\ell^a \Z$.  Since $\ell^a \in I$, we may regard 
$I$ as an ideal of $\OO/\ell^a \OO$ such that $(\OO/\ell^a \OO)/I \cong \Z/\ell^a \Z$.  In other words, we have a $\Z/\ell^a\Z$-algebra 
homomorphism 
\[ f \colon \Z/\ell^a \Z[t]/(t^2-D) \ra \Z/\ell^a \Z. \]
Then $f(t)^2 = D \in \Z/\ell^a \Z$, so $D$ is a square in $\Z/\ell^a\Z$.  \\
Case 2 ($\ell = 2$, $\Delta$ is odd): Then $\left( \frac{\Delta}{\ell} \right) = \pm 1$.  \\
$\bullet$ If $\left( \frac{\Delta}{\ell} \right) = 1$, 
then $\Delta \equiv 1 \pmod{8}$; by Hensel's Lemma, $\Delta$ is a square in $\Z/\ell^a \Z$.  On the other hand, by Lemmas \ref{NEWLEMMA2}a) and \ref{NEWLEMMA4}, $H(\OO,\ell^a)$ holds.  \\
$\bullet$ If $\left( \frac{\Delta}{\ell} \right) = -1$, then $\Delta \equiv 5 \pmod{8}$, so $\Delta$ is not a square modulo $8$ 
and thus not a square modulo $4 \cdot 2^a$.   On the other hand, by Lemma \ref{NEWLEMMA2}b) $H(\OO,\ell^a)$ does not hold.  \\
Case 3: ($\ell = 2$, $\Delta$ is even):  Again we may put $D = \frac{\Delta}{4} \in \Z/\ell^a \Z$, and again (\ref{QUADORDEREQ}) holds.  The argument of Case 1 shows that $H(\OO,\ell^a)$ holds iff $D$ is a square modulo $\Z/\ell^a \Z$ iff $\Delta$ is a square modulo $\Z/4 \ell^a \Z$.
\end{proof}

\begin{prop}
\label{BIGGESTCARTANPROP}
Let $\OO$ be an order, and let $N \in \Z^+$.  The following are equivalent:  \\
(i) $C_N(\OO)$ acts simply transitively on order $N$ elements of $\OO/N\OO$.  \\
(ii) $C_N(\OO)$ acts transitively on order $N$ elements of $\OO/N\OO$.  \\
(iii) For all primes $\ell \mid N$ we have $\left( \frac{\Delta}{\ell} \right) = -1$.
\end{prop}
\begin{proof}
As usual, we may assume $N = \ell^a$ is a prime power.  Certainly (i) $\implies$ (ii).  \\
(ii) $\implies$ (iii): We have 
\[ \#C_{\ell^a}(\OO) = \ell^{2a-2} (\ell-1)\left(\ell - (\frac{\Delta}{\ell})\right), \]
whereas the number of elements of order $\ell^a$ in $\OO/\ell^a\OO$ is 
\[ N(\OO,\ell^a) \coloneqq \# \OO/\ell^a \OO - \# \ell \OO/\ell^{a} \OO = \ell^{2a-2} (\ell-1)(\ell+1). \]
Transitivity of the action implies $\# C_{\ell^a}(\OO) \geq N(\OO,\ell^a)$, which holds iff $\left( \frac{\Delta}{\ell} \right) = -1$. \\
(iii) $\implies$ (i): Since $\left( \frac{\Delta}{\ell} \right) \neq 0$, we have $\OO/\ell^a \OO \cong 
\OO_K/\ell^a \OO_K$, and thus also $C_{\ell^a}(\OO) = (\OO/\ell^a \OO)^{\times} \cong C_{\ell^a}(\OO_K)$.  Thus $\OO/\ell^a \OO$ is a finite local principal ring with maximal 
ideal $\mm = \langle \ell \rangle$ and unit group $C_{\ell^a}(\OO) = \OO/\ell^a \OO \setminus \mm$.   The set of order 
$\ell^a$ elements of $\OO/\ell^a \OO$ is $\OO/\ell^a\OO \setminus \mm = C_{\ell^a}(\OO)$, so the action of the 
unit group $C_{\ell^a}(\OO)$ on this set is the action of $C_{\ell^a}(\OO)$ on itself, which is simply transitive.
\end{proof}

\begin{cor}
Let $\OO$ an order of conductor $\ff$.  Let $N = \prod_{i=1}^r \ell_i^{a_i} \in \Z^+$ be such that $\left( \frac{\Delta}{\ell_i} \right) = -1$ for all $i$.  Let $F$ be a number field, and let $E_{/F}$ be an $\OO$-CM elliptic curve
such that $E(F)$ has a point of order $N$.  
Then 
\begin{equation}
\label{BIGGESTSPYEQ}
 \# \overline{C_N(\OO)} \mid [FK:K(\ff)]. 
\end{equation}
Moreover, for all $\OO$ and $N$ satisfying the above conditions, equality can occur in (\ref{BIGGESTSPYEQ}).  
\end{cor}
\begin{proof}
Replace $F$ by $FK$; then $F \supset K(\ff)$.  By Proposition \ref{BIGGESTCARTANPROP}, 
$C_N(\OO)$ acts transitively on order $N$ elements of $\OO/N\OO$, so the $\OO$-submodule generated by any one of them is $\OO/N\OO$.    Thus the existence of one $F$-rational point of order $N$ implies that $\rho_N$ is trivial.   Applying Theorem \ref{MAINTHM} gives (\ref{BIGGESTSPYEQ}).  That equality 
can occur follows from Corollary \ref{1.1C}.
\end{proof}

\subsection{Torsion over $K(j)$: Part I}
Let $\OO$ be an order of discriminant $\Delta = \ff^2 \Delta_K$.  We will give a complete classification of the possible torsion subgroups of $\OO$-CM elliptic curves $E_{/K(\ff)}$.  In this section we will treat the cases $\Delta \neq -3,-4$.  For the remaining cases we will make use of Theorem \ref{BIGONE}, so we will come back to those cases in $\S$7.5.
\\ \\
If $E(K(\ff))$ has a point of order $N$, 
then since $[C_N(\OO):\rho_N(\gg_{K(\ff)})] \mid \# \OO^{\times}$, there must be some $P \in \OO/N\OO$ of order $N$ with a $C_N(\OO)$-orbit of order dividing $\# \OO^{\times}$.  
\\ \\
$\bullet$ By Theorem \ref{SPY}, if $E(K(\ff))$ has a point of order $N$, then $\varphi(N) \mid 2$, so 
\[N \in \{1,2,3,4,6\}. \]
$\bullet$ Lemma \ref{LASTLEMMA}b) implies that for all $N \geq 3$, we have $\# C_N(\OO) \geq 4$ (equality holds 
if $N = 3$ and $\Delta \equiv 1 \pmod{3}$).  By Theorem \ref{MAINTHM} we cannot have $E[N] = E[N](K(\ff))$.  
\\ \\
Thus $E(K(\ff))[\tors]$ is isomorphic to one of the groups in the following list: 
\\ 
\[ \{e\}, \Z/2\Z, \Z/3\Z, \Z/4\Z, \Z/6\Z, \Z/2 \Z \times \Z/2 \Z, \Z/2\Z \times \Z/4\Z, \Z/2\Z \times \Z/6\Z. \]
We will show that all of these groups occur.
\\ \\
\textbf{Points of order 2}: By Theorem \ref{MAINTHM}, $E(K(\ff))[2]$ has order $4$ if 
$2$ splits in $\OO$, order $2$ if $2$ ramifies in $\OO$ and order $1$ if $2$ is inert in $\OO$.  Thus:
\[ E(K(\ff))[2] \cong 
\begin{cases}
\{e\} & \Delta \equiv 5 \pmod{8} \\
 \Z/2\Z      & \Delta \equiv 0 \pmod{4} \\
\Z/2\Z \times \Z/2\Z & \Delta \equiv 1 \pmod{8}
\end{cases}. \]
\textbf{Points of order 3, 4, or 6}: Let $E_{/K(\ff)}$ be any $\OO$-CM elliptic curve.  We claim that for $N \in \{3,4,6\}$, there is a 
quadratic twist $E^D$ of $E$ such that $E^D(K(\ff))$ has a point of order $N$ iff $H(\OO,N)$ holds.  Indeed, as above, 
since the index of the mod $N$ Galois representation in $C_N(\OO)$ divides $2$, if some $E^D(K(\ff))$ has a point of order $N$, 
then $\O/N\OO$ has a point of order $N$ with a $C_N(\OO)$-orbit of size $2$.  Since $\varphi(N) = 2$, there is a 
Cartan orbit of size $2$ iff $H(\OO,N)$ holds.  Conversely, if $H(\OO,N)$ holds then there is a point of order $N$ with a 
$C_N(\OO)$-orbit of size $2$, hence on some quadratic twist $E^D$ we have an $F$-rational point of order $N$.  Applying 
Theorem \ref{SQUAREDISCTHM}, we get: 
\\  \\
$\bullet$ Some $\OO$-CM $E_{/K(\ff)}$ has a point of order $3$ iff $\Delta \equiv 0,1 \pmod{3}$.  \\
$\bullet$ Some $\OO$-CM $E_{/K(\ff)}$ has a point of order $4$ iff $\Delta \equiv 0,1,4,9 \pmod{16}$. \\
$\bullet$ Some $\OO$-CM $E_{/K(\ff)}$ has a point of order $6$ iff $\Delta \equiv 0,1,2,9,12,16 \pmod{24}$.  
\\ \\ 
Because the only full $N$-torsion we can have is full $2$-torsion, and $2$-torsion is invariant under quadratic 
twists, we immediately deduce the complete answer in all cases.
\\ \\
$\bullet$ If $\Delta \equiv 0 \pmod{48}$, then there are twists $E_1,E_2,E_3$ of $E$ with 
\[ E_1(K(\ff))[\tors] \cong \Z/2\Z, \ E_2(K(\ff))[\tors] \cong \Z/4\Z, \ E_3(K(\ff))[\tors] \cong \Z/6\Z. \]
$\bullet$ If $\Delta \equiv 1,9,25,33 \pmod{48}$ then there are twists $E_1,E_2$ of $E$ with 
\[ E_1(K(\ff))[\tors] \cong \Z/2\Z \times \Z/2\Z, \  E_2(K(\ff))[\tors] \cong \Z/2\Z \times \Z/6\Z. \]
$\bullet$ If $\Delta \equiv 4,16,36 \pmod{48}$, then there are twists $E_1,E_2,E_3$ of $E$ with 
\[ E_1(K(\ff))[\tors] \cong \Z/2\Z, \ E_2(K(\ff)) \cong \Z/4\Z, \ E_3(K(\ff)) \cong \Z/6\Z. \]
$\bullet$ If $\Delta \equiv 5,29 \pmod{48}$, then $E(K(\ff))[\tors] = \{e\}$. \\
$\bullet$ If $\Delta \equiv 8,44 \pmod{48}$, then $E(K(\ff))[\tors] = \Z/2\Z$. \\
$\bullet$ If $\Delta \equiv 12,24,28,40 \pmod{48}$, then there are twists $E_1,E_2$ of $E$ with 
\[ E_1(K(\ff))[\tors] \cong \Z/2\Z, \ E_2(K(\ff)) \cong \Z/6\Z. \]
$\bullet$ If $\Delta \equiv 13,21,37,45 \pmod{48}$, then there are twists $E_1,E_2$ of $E$ with 
\[E_1(K(\ff))[\tors]  = \{e\}, \ E_2(K(\ff))[\tors] \cong \Z/3\Z. \]
$\bullet$ If $\Delta \equiv 17,41 \pmod{48}$, then 
there are twists $E_1,E_2$ of $E$ with 
\[ E_1(K(\ff))[\tors] \cong \Z/2\Z \times \Z/2\Z, \ E_2(K(\ff))[\tors] \cong \Z/2\Z \times \Z/4\Z. \]
$\bullet$ If $\Delta \equiv 20,32 \pmod{48}$, then there are twists $E_1,E_2$ of $E$ with 
\[ E_1(K(\ff))[\tors] \cong \Z/2\Z, \ E_2(K(\ff)) \cong \Z/4\Z. \]

\subsection{Isogenies over $K(j)$: Part I}

\begin{thm}
\label{KJISOGTHM}
Let $\OO$ be an order of discriminant $\Delta = \ff^2 \Delta_K$, and let $N \in \Z^+$.   \\
a) If $\Delta \neq -3,-4$, then there is an $\OO$-CM elliptic curve $E_{/K(\ff)}$ with a $K(\ff)$-rational cyclic $N$-isogeny 
iff $\Delta$ is a square in $\Z/4N\Z$.  \\
b) If $\Delta = -4$, then then there is an $\OO$-CM elliptic curve $E_{/K(\ff)}$ with a $K(\ff)$-rational cyclic $N$-isogeny 
iff $N$ is of the form $2^{\epsilon} \ell_1^{a_1} \cdots \ell_r^{a_r}$ for primes $\ell_i \equiv 1 \pmod{4}$ and $\epsilon,a_1,\ldots,a_r \in \mathbb{N}$ with $\epsilon \leq 2$.  \\
c) If $\Delta = -3$, then there is an $\OO$-CM elliptic curve $E_{/K(\ff)}$ with a $K(\ff)$-rational cyclic $N$-isogeny 
iff $N$ is of the form $2^{\epsilon} 3^a \ell_1^{a_1} \cdots \ell_r^{a_r}$ for primes $\ell_i \equiv 1 \pmod{3}$,
$\epsilon,a,a_1,\ldots,a_r \in \mathbb{N}$ with $(\epsilon,a) \in \{ (0,0), \ (0,1), \ (0,2), \ (1,0), \ (1,1)\}$.  
\end{thm}
\begin{proof}
Step 1: Let $E_{/K(\ff)}$ be an $\OO$-CM elliptic curve.  If $\Delta$ is a square in $\Z/4N\Z$, then by Theorem \ref{SQUAREDISCTHM} 
there is a point $P$ of order $N$ in $\OO/N\OO$ such that $C = \langle P \rangle$ is invariant under $C_N(\OO)$, so $C$ 
is $\gg_{K(\ff)}$-stable and $E \ra E/C$ is a cyclic $N$-isogeny.  If $\Delta \notin \{-4,-3\}$, then the projective Galois representation $\PP \rho_N \colon \gg_{K(\ff)} \ra C_N(\OO)/(\Z/N\Z)^{\times}$
is a quotient of the reduced Galois representation, hence surjective.  So $K(\ff)$-rational cyclic $N$-isogenies correspond to $C_N(\OO)$-orbits on $\OO/N\OO$ of size $\varphi(N)$, 
which by Theorem \ref{SQUAREDISCTHM} exist iff $\Delta$ is a square in $\Z/4N\Z$.  \\
Step 2: If $\Delta \in \{-4,-3\}$, then as above the condition that $\Delta$ is a square modulo $4N$ is 
sufficient for the existence of a $K(\ff)$-rational cyclic $N$-isogeny, but it is no longer clear that it is necessary, and in both 
cases it turns out not to be.  The complete analysis will make use of Theorem \ref{BIGONE}, so we defer the end of the proof 
until $\S$7.6.
\end{proof}

\section{The Torsion Degree Theorem}

\subsection{Statement and Preliminary Reduction}
Throughout this section $\OO$ denotes an order of conductor $\ff$ and discriminant $\Delta = \ff^2 \Delta_K$.  
\\ \\
For $N \in \Z^{\geq 2}$, let 
$\widetilde{T}(\OO,N)$ be the least size of an orbit of $C_N(\OO)$ on an order $N$ point of $\OO/N\OO$.

\begin{lemma}
\label{LAST2LEMMA}
We have $\widetilde{T}(\OO,2) = \begin{cases} 1 & \text{if }  \left( \frac{\Delta}{2}\right) \neq -1\\ 3 & \text{if } \left( \frac{\Delta}{2} \right) = -1 
\end{cases}$.
\end{lemma}
\begin{proof}
By Theorem \ref{SQUAREDISCTHM}, we have $\left(\frac{\Delta}{2}\right) \neq -1$ iff there is a $C_2(\OO)$-orbit 
of size $\varphi(2) = 1$ on $\OO/2\OO$ iff $\widetilde{T}(\OO,2) = 1$.  In the remaining case $\left(\frac{\Delta}{2} \right) = -1$
we have $\# C_2(\OO) =3$ and no orbit of size $1$, hence $\widetilde{T}(\OO,2) = 3$.
\end{proof}

\begin{thm}
\label{BIGONE}(Torsion Degree Theorem)
Let $\OO$ be an order of conductor $\ff$, and let $N \in \Z^{\geq 3}$.  \\
a) There is $T(\OO,N) \in \Z^+$ such that: \\
(i) if $F \supset K(\ff)$ is a number field and $E_{/F}$ is an $\OO$-CM elliptic curve with an $F$-rational point of 
order $N$, then $T(\OO,N) \mid [F:K(\ff)]$, and \\
(ii) there is a number field $F \supset K(\ff)$ with $[F:K(\ff)] = T(\OO,N)$ and an $\OO$-CM elliptic curve 
$E_{/F}$ with an $F$-rational point of order $N$.  \\
b) If $(\Delta,N) = (-3,3)$, then $T(\OO,N) = 1$.  \\
c) Suppose $(\Delta,N) \neq (-3,3)$.  Let $N = \ell_1^{a_1} \cdots \ell_r^{a_r}$ 
be the prime power decomposition of $N$.  Then 
\[ T(\OO,N) = \frac{\prod_{i=1}^r \widetilde{T}(\OO,\ell_i^{a_i})}{\# \OO^{\times}}. \]
d) If $\ell^a = 2$, then $\widetilde{T}(\OO,\ell^a) = 2$ is computed in Lemma \ref{LAST2LEMMA}.  If 
$\ell^a>2$, then $\widetilde{T}(\OO,\ell^a)$ is as follows, where $k=\ord_\ell(\ff)$:
\begin{enumerate}
\item If $\ell \nmid \ff$, then
 $\widetilde{T}(\OO,\ell^a) = \begin{cases} \ell^{a-1}(\ell-1) & \text{if } \left( \frac{\Delta}{\ell}\right) = 1,\\ 
\ell^{2a-2}(\ell-1) 
& \text{if } \left( \frac{\Delta}{\ell} \right) = 0, \\ \ell^{2a-2}(\ell^2-1) &  \text{if }\left( \frac{\Delta}{\ell} \right) = -1.
\end{cases}$
\vspace{.2cm}
\item If $\ell \mid \ff$, then 
 $\widetilde{T}(\OO,\ell^a) = \begin{cases} \ell^{a-1} (\ell-1) & \text{if } \left( \frac{\Delta_K}{\ell}\right) = 1,\\  \ell^{a-1}(\ell-1) & \text{if } \left( \frac{\Delta_K}{\ell}\right) = -1 \text{ and } a \leq 2k,\\ \ell^{2a-2k-1}(\ell-1) &  \text{if } \left( \frac{\Delta_K}{\ell}\right) = -1 \text{ and } a > 2k, \\ \ell^{a-1}(\ell-1) & \text{if } \left( \frac{\Delta_K}{\ell}\right) = 0 \text{ and } a \leq 2k+1, \\ 
\ell^{2a-2k-2}(\ell-1)
 & \text{if } \left( \frac{\Delta_K}{\ell}\right) = 0 \text{ and } a > 2k+1.
\end{cases}$
\end{enumerate}
\end{thm}

\begin{remark} 
The case $N = 2$ is excluded because of the somewhat anomalous behavior of $2$-torsion.  But it is easy to see 
that Theorem \ref{BIGONE}a) remains true when $N = 2$, and moreover: \\
$\bullet$ If $\Delta \in \{-4,-3\}$ then $T(\OO,2) = 1$.  \\
$\bullet$ Otherwise, $T(\OO,2) = \begin{cases} 1 & \text{if }  \left(\frac{\Delta}{2} \right) \neq -1 \\ 3 & \text{if }  \left( \frac{\Delta}{2} \right) = -1 
\end{cases}$.
\end{remark}
\noindent

Let $F \supset K(\ff)$ be a number field, and let $E_{/F}$ be an $\OO$-CM elliptic curve.  As usual, we choose an embedding 
$F \hookrightarrow \C$ such that $j(E) = j(\C/\OO)$.  Let $P \in E[\tors]$ have order $N$.  We call the field
\[ K(\ff)(\hh(P)) \]
the \textbf{field of moduli} of $P$.  It is independent of the chosen model of $E_{/F}$, and there exists an elliptic curve $E'_{/K(\ff)(\hh(P))}$ with an isomorphism $\psi \colon E \rightarrow E'$ such that $\psi(P)$ is $K(\ff)(\hh(P))$-rational.  Further, the pair $(E,P)$ induces a closed point $\mathcal{P}$ on the modular curve $X_1(N)_{/K}$, and $K(\ff)(\hh(P))$ is the residue field $K(\mathcal{P})$.   Theorem \ref{BIGONE} concerns 
the degree $[K(\ff)(\hh(P)):K(\ff)]$.  Our setup shows that it is no loss of generality to assume $F = K(\ff)$.  
\\ \indent
Let $q_N \colon \OO \ra \OO/N\OO$ be the natural map, and let $q_N^{\times} \colon \OO^{\times} \ra C_N(\OO)$ be the induced map on unit groups.  As in the introduction, we define the \textbf{reduced mod N Cartan subgroup}:  
\[\overline{C_N(\OO) }= C_N(\OO)/q_N(\OO^{\times}). \]
Let $\overline{E[N]}$ be the set of $\OO^{\times}$-orbits on $E[N]$. Then the action of $C_N(\OO)$ on $E[N]$ induces an action 
of $\overline{C_N(\OO)}$ on $\overline{E[N]}$.  The field of moduli $K(\ff)(\hh(P))$ depends only on the image $\overline{P}$ of 
$P$ in $\overline{E[N]}$.  By Theorem \ref{MAINTHM}, the composite homomorphism 
\[ \gg_F \stackrel{\rho_{E,N}}{\longrightarrow} C_N(\OO) \ra \overline{C_N(\OO)} \]
is surjective (and model-independent).  Let $H_{\overline{P}} = \{g \in \overline{C_N(\OO)} \mid g \overline{P} = \overline{P}\}$.  It follows that
\[ \Aut(K(\ff)(\hh(P))/K(\ff))\cong \overline{C_N(\OO)}/H_{\overline{P}}. \]
Thus $[K(\ff)(\hh(P)):K(\ff)]$ is the size of the orbit of the reduced Cartan subgroup $\overline{C_N(\OO)}$ on 
$\overline{P}$.  (As we will see, in almost every case this is the size of the orbit of $C_N(\OO)$ on $P$ divided by $\# \OO^{\times}$.)  This reduces the proof of Theorem \ref{BIGONE} to a purely algebraic problem.

\subsection{Generalities}
For an order $N$ point $P \in \OO/N\OO$, let $M_P = \{ x P \mid x \in \OO\}$ be the cyclic $\OO$-submodule of 
$\OO/N\OO$ generated by $P$.  If we put 
$I_P = \{ x \in \OO \mid x P = 0\}$, then we have
\[ M_P\cong_{\OO} \OO/ I_P. \]
The isomorphism is canonical and determined by mapping $P \in M_P$ to $1 + I_P \in \OO/I_P$.

\begin{lemma}
\label{7.1}
a) With notation as above, let \[S(I_P) = \{g \in C_N(\OO) \mid g \equiv 1 \pmod{I_P} \}. \]  Then with respect to the $C_N(\OO)$-action, 
$S(I_P)$ is the stabilizer of $P$, so as a $C_N(\OO)$-set the orbit of $C_N(\OO)$ on $P$ is isomorphic to $C_N(\OO)/S(I_P)$.  \\
b) Moreover, there is a canonical isomorphism of groups $C_N(\OO)/S(I_P) \stackrel{\sim}{\ra} (\OO/I_P)^{\times}$.  
\end{lemma}
\begin{proof}
a) For $g \in C_N(\OO)$, we have $gP = P \iff (g-1)P = 0 \iff (g-1) \in I_P$, giving the first assertion.  The Orbit Stabilizer Theorem gives the second assertion.  \\
b) The ring homomorphism $f \colon \OO/N\OO \ra \OO/I$ induces a homomorphism on unit groups $f^{\times} \colon C_N(\OO) \ra (\OO/I_P)^{\times}$, with kernel $S(I_P)$.   Since $\O/N\OO$ has finitely many maximal ideals, $f^{\times}$ 
is surjective \cite[Thm. 4.32]{Clark-CA}.
\end{proof}

\begin{lemma}
\label{6.4}
There is a positive integer $M \mid N$ such that 
\[ \OO/I_P \cong_{\Z} \Z/N\Z \oplus \Z/M\Z. \]
\end{lemma}
\begin{proof}
As a $\Z$-module, $\OO/I_P$ is a quotient of $\OO/N\OO \cong_{\Z} \Z/N\Z \oplus \Z/N\Z$, so 
\[ \OO/I_P \cong_{\Z} \Z/N' \Z \oplus \Z/M \Z \]
with $M \mid N' \mid N$.  Since $P$ has order $N$ in $(\OO/I_P,+)$, we have $N' = N$. 
\end{proof}
\noindent
The following result computes the size of the reduced Cartan orbit on an order $N$ point of $\OO/N\OO$ in terms of the size of 
the Cartan orbit.  We recall that we have assumed $N \geq 3$.  

\begin{lemma}
\label{CARTANREDUCTIONLEMMA}
a) Suppose $(\Delta,N) \neq (-3,3)$, and let $P \in \OO/N\OO$ have order $N$.  Then the orbit of $C_N(\OO)$ on $P$ 
has size $\# \OO^{\times}$ times the size of the orbit of $\overline{C_N(\OO)}$ on $\overline{P}$. \\
b) Suppose $(\Delta,N) = (-3,3)$.  Then the order $3$ points of $\OO/3\OO$ lie in two orbits under $C_3(\OO)$: one of 
size $2$ and one of size $6$.  The corresponding reduced Cartan orbits each have size $1$.  
\end{lemma}
\begin{proof}
a) The Cartan orbit has size $\# (\OO/I_P)^{\times}$, and the reduced Cartan orbit is smaller by a factor of the cardinality 
of the image of $\OO^{\times} \ra (\OO/I_P)^{\times}$. \\
$\bullet$ Suppose $\Delta \notin \{-4,-3\}$.  Then $\OO^{\times} = \{ \pm 1\}$, 
and since $N \geq 3$, we have $-1 \not \equiv 1 \pmod{I_P}$.  \\
$\bullet$ Suppose $\Delta = -4$.  Since $I_P \not \supset (2)$, by Lemma \ref{CFTLEMMA} the group $U_{I_P}(K)$ is trivial, 
and thus the map $\OO^{\times} \ra (\OO/I_P)^{\times}$ is injective. \\
$\bullet$ Suppose $\Delta = -3$.  By assumption $N \geq 4$, so $I_P \nmid (\zeta_3-1)$ and the map 
$\OO^{\times} \ra (\OO/I_P)^{\times}$ is injective. \\
b) The assertion about Cartan orbits is a case of \cite[Lemma 19]{TORS1}.  (And another proof will be given in the next section.) The fact that both reduced Cartan orbits have size $1$ follows from the already established fact that there is an $\OO$-CM $E_{/\Q(\sqrt{-3})}$ with full $3$-torsion. 
\end{proof}
\noindent
In view of Lemma \ref{CARTANREDUCTIONLEMMA}, to prove Theorem \ref{BIGONE} it suffices to compute the least size of an 
orbit of $C_N(\OO)$ on an order $N$ point of $\OO/N\OO$ and show that this divides the size of every such orbit.  
The following result further reduce us to the case of $N$ a prime power.

\begin{prop}
Let $N \geq 2$ have prime power decomposition $N = \ell_1^{a_1} \cdots \ell_r^{a_r}$.  Let $P \in \OO/N\OO$ have order $N$, 
and let $I_P = \ann P$.  For $1 \leq i \leq r$, let $P_i = \frac{N}{\ell_i^{a_i}} P$, and let $I_{P_i} = \ann P_i$.  Then: \\
a) The ideals $I_{P_1},\ldots,I_{P_r}$ are pairwise comaximal: we have $I_{P_i} + I_{P_j} = \OO$ for all $i \neq j$.  \\
b) We have $I_P = I_{P_1} \cdots I_{P_r}$. \\
c) We have a canonical isomorphism of rings 
\[ \OO/I_P \stackrel{\sim}{\ra} \prod_{i=1}^r \OO/I_{P_i} \]
which induces a canonical isomorphism of unit groups 
\[ (\OO/I_P)^{\times} \stackrel{\sim}{\ra} \prod_{i=1}^r (\OO/I_{P_i})^{\times}. \]
d) The Cartan orbit of $P$ is isomorphic, as a $C_N(\OO)$-set, to the direct product of the $C_{\ell_i^{a_i}}(\OO)$-orbits 
of the $P_i$'s.
\end{prop}
\begin{proof}
a) For $1 \leq i \leq r$, we have $(\OO/I_{P_i},+) \cong \Z/\ell_i^{a_i}\Z \oplus \Z/\ell_i^{b_i} \Z$ with $0 \leq b_i \leq a_i$; 
in particular it is an $\ell_i$-group.  Thus for $i \neq j$, $(\OO/(I_i + I_j),+)$ is a homomorphic image of an $\ell_i$-group 
and an $\ell_j$-group, so it is trivial.  \\
b) By the Chinese remainder theorem, we have $I_{P_1} \cdots I_{P_n} = \bigcap_{i=1}^n I_{P_i}$.  Since $P_i$ is a multiple of $P$, we have $I_P \subset I_{P_i}$ for all $i$, and thus $I_P \subset \bigcap_{i=1}^r I_{P_i}$.  Conversely, choose $y_1,\ldots,y_r \in \Z$ such that $\sum_{i=1}^r y_i \frac{N}{\ell_i^{a_i}} = 1$.  If $x \in \bigcap_{i=1}^r I_{P_i}$ 
then $x \frac{N}{\ell_i^{a_i}} P = 0$ for all $i$, hence 
\[ 0 = \sum_{i=1}^r y_i \frac{N}{\ell_i^{a_i}} xP = xP, \]
so $x \in I_P$.   Thus $I_P = \bigcap_{i=1}^n I_{P_i} = I_{P_1} \cdots I_{P_n}$.  \\
c) The Chinese remainder theorem gives the first isomorphism; the second follows by passing to unit groups. \\
 d) Apply Lemma \ref{7.1} and part c).
\end{proof}

\subsection{The Case $\ell \nmid \ff$}

\begin{thm}
\label{THM7.8}
Let $E_{/K(\ff)}$ be an $\OO$-CM elliptic curve.  Let $\ell^a > 2$ be a prime power such that $\ell \nmid \ff$.  We will describe all orbits of $C_{\ell^a}(\OO)$ on order $\ell^a$ points of $\OO/\ell^a\OO$: their sizes and their multiplicities.  \\
a) If $\left( \frac{\Delta}{\ell} \right) = 1$, there are $2a+1$ orbits: two orbits of size $\ell^{a-1}(\ell-1)$, for all 
$1 \leq i \leq a-1$ two orbits of size $\ell^{a+i-2}(\ell-1)^2$, and one orbit of size $\ell^{2a-2}(\ell-1)^2$. \\
b) If $\left( \frac{\Delta}{\ell} \right) = 0$, there are two orbits: an orbit of size $\ell^{2a-2}(\ell-1)$ and an orbit of size 
$\ell^{2a-1}(\ell-1)$.  \\
c) If $\left( \frac{\Delta}{\ell} \right) = -1$, there is one orbit, of size $\ell^{2a-2} (\ell^2-1)$.
\end{thm}
\begin{proof}
Step 1: Suppose $\OO = \OO_K$.  Then every $\OO$-submodule of $E[N]$ is 
of the form $E[I]$ for an ideal $I \supset N \OO$, and $E[I] \cong_{\OO} \OO/I$: thus every 
submodule is of the form $M_P = \langle P \rangle_{\OO}$ and is determined by its annihilator ideal $I_P$.  Conversely, 
if $I \supset N\OO$ is an ideal, then Lemmas \ref{LAZARUSLEMMA} and \ref{INVLEMMA1} give that $E[I]$ is an $\OO$-submodule of $E[N]$ with annihilator ideal $I$.  
 \\
\textbf{Split Case $\left( \frac{\Delta}{\ell} \right) = 1$:} Then $\ell \OO = \pp_1 \pp_2$ for distinct prime ideals $\pp_1, \pp_2$ 
of norm $\ell$.  The ideals containing $\ell^a \OO$ are precisely $\pp_1^c \pp_2^d$ with $\max(c,d) \leq a$.  We have ring isomorphisms
\[ \OO/\pp_1^c \pp_2^d \cong \OO/\pp_1^c \times \OO/\pp_2^d \cong \Z/\ell^c \Z \times \Z/\ell^d \Z, \]
hence unit group isomorphisms
 \[ (\OO/\pp_1^c \pp_2^d)^{\times} \cong (\OO/\pp_1^c)^{\times} \times (\OO/\pp_2^c)^{\times} \cong 
(\Z/\ell^c \Z)^{\times} \times (\Z/\ell^d \Z)^{\times}, \]
so
\[ \# (\OO/\pp_1^c \pp_2^d)^{\times} = \varphi(\ell^c) \varphi(\ell^d). \]
To get points of order $\ell^a$ we impose the condition $\max(c,d) = a$.  Thus $\OO$-modules generated by the points of order $\ell^a$ are 
\[ E[\pp_1^a], \  E[\pp_1^a \pp_2], \ldots, \ E[\pp_1^a \pp_2^a] = E[\ell^a], \ E[\pp_1^{a-1} \pp_2^a], \ldots, \ E[\pp_1 \pp_2^a], \ E[\pp_2^a]. \]
So there are $2a+1$ Cartan orbits, one of size $\varphi(\ell^a) \varphi(\ell^a)$ and, for all $0 \leq i \leq a-1$, two of size 
$\varphi(\ell^a) \varphi(\ell^i)$. The smallest orbit size is $\ell^{a-1}(\ell-1)$, and all the other orbit 
sizes are multiples of it.
\\ 
\textbf{Ramified Case $\left( \frac{\Delta}{\ell} \right) = 0$:} Then $\ell \OO = \pp^2$ for a prime ideal $\pp$ of norm $\ell$.  
 For any $b \in \Z^+$, the ring $\OO/\pp^b$ is local 
of order $\ell^b$ with residue field $\Z/\ell\Z$, so the maximal ideal has size $\ell^{b-1}$ and thus
\[ \# (\OO/\pp^b)^{\times} = \ell^b - \ell^{b-1} = \ell^{b-1}(\ell-1). \]
Since $\pp^2 = (\ell)$, the least $c \in \mathbb{N}$ such that $\ell^c \in \pp^b$ is $c = \lceil \frac{b}{2} \rceil$.  It follows that
\[ (\OO/\pp^b,+) \cong_{\Z} \Z/\ell^{\lceil \frac{ b}{2} \rceil}\Z \oplus \Z/\ell^{\lfloor \frac{b}{2} \rfloor} \Z. \]
So the annihilator ideals of points of order $\ell^a$ in $\OO/\ell^a \OO$ are precisely $\pp^{2a-1}$ and $\pp^{2a}$.  We 
get two Cartan orbits, one of size $\# (\OO/\pp^{2a-1})^{\times} = \ell^{2a-2}(\ell-1)$ and one of size $\# (\OO/\pp^{2a})^{\times} = \ell^{2a-1}(\ell-1)$.  
The smallest orbit size is $\ell^{2a-2}(\ell-1)$, and the other orbit size is a multiple of it.  
\\ 
\textbf{Inert Case $\left( \frac{\Delta}{\ell} \right) = -1$}: Then $\ell \OO$ is a prime ideal, so the ideals containing $\ell^a \OO$ are 
precisely $\ell^i \OO$ for $i \leq a$.  Clearly $\OO/\ell^i \OO$ has exponent $\ell^a$ iff $i = a$, so the $\OO$-module generated 
by any point of order $\ell^a$ is $E[\ell^a]$.  There is a single Cartan orbit, of size $\# (\OO/\ell^a \OO)^{\times} = \varphi_K(\ell^a) = 
\ell^{2a-2}(\ell^2-1)$.  \\
Step 2: Now let $\OO$ be an order with $\ell \nmid \ff$.  The natural maps $\OO/\ell^a \OO \ra \OO_K/\ell^a \OO_K$ and 
$C_{\ell^a}(\OO) \ra C_{\ell^a}(\OO_K)$ are isomorphisms, so the sizes and multiplicities of orbits carry over from $\OO_K$ 
to $\OO$. 
\end{proof}

\subsection{The Case $\ell \mid \ff$}
Now suppose $\ell \mid \ff$.  The ring $\OO/\ell \OO$ is isomorphic to $\Z/\ell\Z[\epsilon]/(\epsilon^2)$ -- as one sees, e.g., using the explicit representation of 
(\ref{ITTEQ1}) -- and is thus a local Artinian ring with maximal ideal $\pp$, say, and residue field $\Z/\ell\Z$.  Because $[\pp:\ell O] = \ell$, the only proper nonzero $\OO$-submodule of $\OO/\ell \OO$ is $\pp/\ell$.  Thus there are two Cartan 
orbits on the order $\ell$ elements of $\OO/\ell \OO$: one of order $\ell-1$ and one of order 
$\ell^2-\ell = \# (\OO/\ell \OO)^{\times}$.  
\\ \indent
For all $a \in \Z^+$, the ring $\OO/\ell^a \OO$ is local -- for a maximal ideal $\mm$ of $\OO$, we have $\ell^a \in \mm \iff \ell \in \mm$ -- 
with residue field $\Z/\ell\Z$.  In turn it follows that for any order $\ell^a$ point $P \in \OO/\ell^a \OO$ and $I_P = \{x \in \OO \mid x P = 0\}$, the ring $\OO/I_P$ is  local with residue field $\Z/\ell\Z$.  By Lemma \ref{6.4}, we may write 
\begin{equation}
\label{CARTANLIFTINGEQ}
M_P = \OO/I_P \cong_{\Z} \Z/\ell^a \Z \oplus \Z/\ell^b \Z 
\end{equation}
for some $0 \leq b \leq a$, and then 
\[ \# (\OO/I_P)^{\times} = \# \OO/I_P - \frac{\# \OO/I_P}{\ell} = \ell^{a+b-1}(\ell-1). \]
So the size of a Cartan orbit on an order $\ell^a$ element of $\OO/\ell^a \OO$ is of the form 
$(\ell-1) \ell^c$ for some $a-1 \leq c \leq 2a-1$.  So in this case it is \emph{a priori} clear that the minimal size of a 
Cartan orbit divides the size of all the Cartan orbits.  We want to understand how Cartan orbits grow when we lift a point of 
order $\ell^a$ to a point of order $\ell^{a+1}$.  First observe that $x \mapsto \ell x$ gives an $\OO$-module isomorphism 
\[\OO/\ell^a \OO \stackrel{\sim}{\ra} \ell \OO/\ell^{a+1} \OO, \]
so we can view $\OO/\ell^a \OO$ as an $\OO$-submodule of $\OO/\ell^{a+1} \OO$.  With $P$ as in (\ref{CARTANLIFTINGEQ}), let $Q \in \OO/\ell^{a+1} \OO$ be such that $\ell Q = P$.  Put $M_Q = \{ x Q \mid x \in \OO\}$ and $I_Q = \{ x \in \OO \mid x Q = 0\}$, and write 
\begin{equation}
\label{CARTANLIFTINGEQ2}
M_Q = \OO/I_Q \cong_{\Z} \Z/\ell^{a+1} \Z \oplus \Z/\ell^{b'} \Z 
\end{equation}
for $0 \leq b' \leq a+1$.  Because $\ell Q = P$, we have $\ell M_Q = M_P$.  Thus we find: if $b = 0$, then $b' \in \{0,1\}$, 
whereas if $b \geq 1$ then necessarily $b' = b+1$.  So: if the $C_{\ell^a}(\OO)$-orbit on $P$ has the smallest possible size $\varphi(\ell^a)$, 
then the $C_{\ell^{a+1}}(\OO)$-orbit on $Q$ either has size $\varphi(\ell^{a+1})$ or size $\varphi(\ell^{a+2})$ (as we will see shortly, 
both possibilities can occur), whereas if the $C_{\ell^a}(\OO)$-orbit on $P$ has size $\varphi(\ell^{a+b}) > \varphi(\ell^a)$, 
then the $C_{\ell^{a+1}}(\OO)$-orbit on $Q$ has size $\varphi(\ell^{a+b+2})$: i.e., upon lifting from $P$ to $Q$ the size 
grows by a factor of $\ell^2$.  
\\ \\
Since $H(\OO,\ell^{a+1})$ implies $H(\OO,\ell^a)$, for each fixed $\ell$ and $\OO$ there are two possibilities.
\\ 
\textbf{Type I}: $H(\OO,\ell^a)$ holds for all $a \in \Z^+$.  \\
In Type I, for all $a \in \Z^+$ the least size of a $C_{\ell^a}(\OO)$-orbit is $\varphi(\ell^{a})$.  
\\ 
\textbf{Type II:} There is some $A \in \Z^+$ such that $H(\OO,\ell^a)$ holds iff $a \leq A$. \\
In Type II, for $1 \leq a \leq A$, the least size of a $C_{\ell^a}(\OO)$-orbit is $\varphi(\ell^a)$, but for all $a \geq A$, 
whenever we lift a point of order $\ell^a$ to a point of order $\ell^{a+1}$ the size of the Cartan orbit grows by a factor of 
$\ell^2$, so for all $a > A$ the least size of a $C_{\ell^a}(\OO)$-orbit is $\ell^{a-A} \varphi(\ell^a)$.  
\\ \\
We now determine the smallest size of a $C_{\ell^a}(\OO)$-orbit on an order $\ell^a$ point of $\OO/\ell^a \OO$ by using Theorem \ref{SQUAREDISCTHM} to determine the type and compute the value of $A$ in Type II.
\\ \\
\textbf{Case 1}: Suppose $\left( \frac{\Delta_K}{\ell} \right) = 1$.  Then for all $a \in \Z^+$ $H(\OO_K,\ell^a)$ holds, so 
$\Delta_K$ is a square modulo $4 \ell^a$, hence $\Delta = \ff^2 \Delta_K$ is also a square modulo $4 \ell^a$, so 
$H(\OO,\ell^a)$ holds, and we are in Type I. 
\\ 
\textbf{Case 2}: Suppose $\left( \frac{\Delta_K}{\ell} \right) = -1$, and put $k = \ord_{\ell}(\ff)$.  \\ \indent
$\bullet$ Let $\ell > 2$.  If $a \leq 2k$, then  $\ell^a \mid \Delta$, so $\Delta$ is a square mod $\ell^a$ and 
hence also mod $4 \ell^a$: thus $H(\OO,\ell^a)$ holds.  However, if $a = 2k+1$ then we claim $H(\OO,\ell^a)$ does not hold.  Indeed, suppose there is $s \in \Z$ such that $\Delta = \ff^2 \Delta_K \equiv s^2 \pmod{\ell^a}$.  Then $\ell^k \mid s$; taking 
$S = \frac{s}{\ell^k}$ we have $\frac{\ff^2}{\ell^{2k}} \Delta_K \equiv S^2 \pmod{\ell^{a-2k}}$, which implies that 
$\Delta_K$ is a square modulo $\ell$: contradiction.  So we are in Type II with $A = 2k$.  \\ \indent
$\bullet$ Let $\ell =2$, and write $\ff = 2^k F$.  Suppose $a \leq 2k$.  Since $4 \mid \Delta_K -1$, we have
\[ 2^{a+2} \mid    (2^k F)^2 (\Delta_K-1) =  \Delta - (2^k F)^2, \]
so $H(\OO,2^a)$ holds.  Suppose $a \geq 2k+1$.  If $\Delta$ is a square modulo $2^{a+2}$, then we find that $\Delta_K \equiv 1 \pmod{8}$, 
so $\left( \frac{\Delta_K}{2} \right) = 1$: contradiction.  So we are in Type II with $A = 2k$. 
\\ 
\textbf{Case 3}: Suppose $\left( \frac{\Delta_K}{\ell} \right) = 0$, and put $k = \ord_{\ell}(\ff)$. \\ \indent
$\bullet$ Let $\ell > 2$.  If $a \leq 2k+1$, then $\ell^a \mid \Delta$, so $\Delta$ is a square mod $\ell^a$ and hence also 
mod $4 \ell^a$: thus $H(\OO,\ell^a)$ holds.  However, if $a = 2k+2$ then we claim $H(\OO,\ell^a)$ does not 
hold.  Indeed, $\ord_{\ell}(\Delta) = 2k+1 < a$, so if $\Delta \equiv s^2 \pmod{\ell^a}$, then $\ord_{\ell}(s^2) = 2k+2$: 
contradiction.   So we are in Type II with $A = 2k+1$. \\ \indent
$\bullet$ Let $\ell = 2$, and write $\ff = 2^k F$.  Suppose $a \leq 2k+1$.  Since $4 \mid \Delta_K$, there is $s \in \Z$ 
such that $8 \mid \Delta_K - s^2$, so 
\[ 2^{a+2} \mid 2^{2k+3} \mid  (2^k F)^2 (\Delta_K - s^2) = \Delta - (2^k Fs)^2, \]
so $H(\OO,2^a)$ holds.  Suppose $a \geq 2k+2$.  If $\Delta$ is a square modulo $2^{a+2}$, then 
$\Delta_K$ is a square modulo $2^{a+2-2k}$, hence modulo $16$: contradiction.  So we are in Type II with $A = 2k+1$.

\subsection{Torsion over $K(j)$: Part II}
We return to complete the classification of torsion on $\OO$-CM elliptic curves $E_{/K(\ff)}$ begun in $\S$6.6.
\\ \\
II. Suppose $\Delta = -4$, so $j = 1728$ and $K(\ff) = K = \Q(\sqrt{-1})$.  
\\ \\
$\bullet$ By Theorem \ref{SPY}, if $E(K)$ has a point of order $N$, then $\varphi(N) \mid 4$, so 
\[N \in \{1,2,3,4,5,6,8,10\}. \]
$\bullet$ Using Theorem \ref{BIGONE} we get 
\[ T(\OO,1) = T(\OO,2) = T(\OO,4) = T(\OO,5) = T(\OO,10) = 1, \]
\[ T(\OO,3) = T(\OO,6) = 2, \ T(\OO,8) = 4. \]
$\bullet$ We have $C_2(\OO) = \mu_4/\{\pm 1\}$.  Thus $\# C_2(\OO) = 2$ so every $\OO$-CM elliptic curve $E_{/K}$ has a 
$K$-rational point of order $2$, and some $\OO$-CM elliptic curve $E_{/K}$ has $E[2] = E[2](K)$.  \\
$\bullet$  Because $\widetilde{T}(\OO,5) = 4$, if an $\OO$-CM elliptic curve $E_{/K}$ has a $K$-rational point of order $5$, 
the index of the mod $5$ Galois representation in $C_5(\OO)$ is divisible by $4$.   Because $\# C_2(\OO) = 2$, 
if an $\OO$-CM elliptic curve $E_{/K}$ has full $2$-torsion then the index of the mod $2$ Galois 
representation in $C_2(\OO)$ is divisible by $2$.  Thus if an $\OO$-CM elliptic curve $E_{/K}$ had $\Z/2\Z \times \Z/10\Z \hookrightarrow E(K)[\tors]$, the index of the mod $10$ Galois representation in $C_{10}(\OO)$ would be divisible by $8$, 
contradicting Corollary \ref{COR1.4}.  \\
$\bullet$ If $N \geq 3$ then $\# C_N(\OO) > \# \OO^{\times}$, so no $\OO$-CM elliptic curve $E_{/K}$ has $E[N] = E[N](K)$.  \\
$\bullet$ If there were a CM elliptic curve $E_{/K}$ with $E(K)[\tors] \cong \Z/4\Z$, then there would be an ideal $I$ of $\OO$ 
such that $\OO/I$ is isomorphic as a $\Z$-module to $\Z/4\Z$.  But there is no such an ideal, a special case of the analysis 
done in the proof of Theorem \ref{THM7.8}.
\\ \\
Thus the groups which can occur 
as $E(K)[\tors]$ are precisely 
\[ \Z/2\Z,  \Z/10\Z, \Z/2\Z \times \Z/2\Z, \Z/2\Z \times \Z/4\Z. \]

\begin{example}
\label{EXAMPLE7.9}
For $K = \Q(\sqrt{-1})$, every $\OO_K$-CM elliptic curve $E_{/K}$ is isomorphic over $K$ to 
\[E_A: y^2 = x^3 + Ax\] 
for some $A \in K^{\times}$.  We exhibit such elliptic curves with all possible torsion subgroups.  

{\footnotesize
\begin{center}
    \begin{tabular}{ c||c}
      $A$ & $E(\Q(\sqrt{-1}))[\tors]$\\  \hline
$2$ & $\Z/2\Z$ \\
$64-128\sqrt{-1}$ & $\Z/10\Z$ \\
$1$ & $\Z/2\Z \times \Z/2\Z$ \\
$4$ & $\Z/2\Z \times \Z/4\Z$ 
     \end{tabular}
\end{center}
}
\noindent
For the groups $\Z/2\Z$, $\Z/2\Z \times \Z/2\Z$ and $\Z/2\Z \times \Z/4\Z$, we have $A \in \Q^{\times}$ and thus $E_A$ arises 
from an elliptic curve defined over $\Q$ via base extension.  Gonz\'alez-Jim\'enez has shown \cite[Thm. 1]{GJ19} that for no $A \in \Q^{\times}$ do we have $E_A(K)[\tors] \cong \Z/10\Z$.  Najman \cite{Najman10}, \cite{Najman11} has classified the torsion subgroups of 
all elliptic curves (CM or otherwise) defined over $K$.  
\end{example}

III. Suppose $\Delta = -3$, so $j = 0$ and $K(\ff) = K = \Q(\sqrt{-3})$.
\\ \\
$\bullet$ By Theorem \ref{SPY}, if $E(K(\ff))$ has a point of order $N$, then $\varphi(N) \mid 6$, so 
\[ N \in \{1,2,3,4,6,7,9,14,18\}. \]
$\bullet$ Using Theorem \ref{BIGONE} we get
\[ T(\OO,1) = T(\OO,2) = T(\OO,3) = T(\OO,6) = T(\OO,7) = 1, \]
\[ T(\OO,4) = 2, \ T(\OO,9) = T(\OO,14) = 3, \ T(\OO,18) = 9. \]
$\bullet$ We have $C_2(\OO) = \mu_6/\{ \pm 1\}$.  Thus as we range 
over all $\OO$-CM elliptic curves $E_{/K}$, the group $E(K)[2]$ can be trivial (using Corollary \ref{LargeTwistCor}) or have size $4$, but it cannot have size $2$.  \\
$\bullet$ We have $C_3(\OO) = \mu_6$.  Thus there is an $\OO$-CM elliptic curve $E_{/K}$ with $E[3] = E[3](K)$. \\
$\bullet$ If $N \geq 4$ then $\# C_N(\OO) > \# \OO^{\times}$, so no $\OO$-CM elliptic curve $E_{/K}$ has 
$E[N] = E[N](K)$.  
\\
Thus the groups which can occur 
as $E(K)[\tors]$ are precisely 
\[ \{e\}, \Z/3\Z, \Z/7\Z, \Z/2\Z \times \Z/2\Z, \Z/2\Z \times \Z/6\Z,  \Z/3\Z \times \Z/3\Z. \]

\begin{example}
\label{EXAMPLE7.10}
For $K = \Q(\sqrt{-3})$, every $\OO_K$-CM elliptic curve $E_{/K}$ is isomorphic over $K$ to 
\[E_A: y^2 = x^3 + B\] 
for some $B \in K^{\times}$.  We exhibit such elliptic curves with all possible torsion subgroups.

{\footnotesize
\begin{center}
    \begin{tabular}{ c||c}
      $B$ & $E(\Q(\sqrt{-1}))[\tors]$\\  \hline
$2$ & $\{e\}$ \\
$4$ & $\Z/3\Z$ \\
$6\sqrt{-3}-54$ & $\Z/7\Z$ \\
$-1$ & $\Z/2\Z \times \Z/2\Z$ \\
$1$ & $\Z/2\Z \times \Z/6\Z$ \\
$16$ & $\Z/3\Z \times \Z/3\Z$ 
     \end{tabular}
\end{center}
}
\noindent
For the groups $\{e\}$, $\Z/3\Z$, $\Z/2\Z \times \Z/2\Z$, $\Z/2\Z \times \Z/6\Z$ and $\Z/3\Z \times \Z/3\Z$, we have 
$B \in \Q^{\times}$ and thus $E_B$ arises from an elliptic curve defined over $\Q$ via base extension.  Again Gonz\'alez-Jim\'enez has shown \cite[Thm. 1]{GJ19} that for no $B \in \Q^{\times}$ do we 
have $E_B(K)[\tors] \cong \Z/7\Z$.   And again Najman \cite{Najman10}, \cite{Najman11} has classified the torsion subgroups of 
all elliptic curves (CM or otherwise) defined over $K$.  
\end{example}

\begin{remark}
a) Case I. of the above calculation is a more detailed and explicit version of one of the main results of \cite{Parish89}.  Parish offers addenda on Cases II. and III., but without proof, and the possibilities $E(K(\ff))[\tors] \cong \Z/10\Z$ 
in Case II. and $E(K(\ff))[\tors] \cong \Z/7\Z$ and $E(K(\ff))[\tors] \cong \Z/3\Z \times \Z/3\Z$ in Case III are not mentioned. \\
b) In Cases II. and III. a classification of the possibilities for $E(K(\ff))[\tors]$ \emph{apart} from the ``Olson groups'' 
$\{e\}$, $\Z/2\Z$, $\Z/3\Z$, $\Z/4\Z$, $\Z/6\Z$, $\Z/2\Z \times \Z/2\Z$ was done in \cite[Thm. 1.4]{BCS} using computer calculations on 
degrees of preimages of $j = 0$ and $j = 1728$ on modular curves \cite[Table 2]{BCS}.  This result was used to find 
$E_A$ with $E_A(\Q(\sqrt{-1})[\tors] \cong \Z/10\Z$ in Example \ref{EXAMPLE7.9} and $E_B$ with $E_B(\Q(\sqrt{-3})[\tors] \cong \Z/7\Z$ 
in Example \ref{EXAMPLE7.10}.
\end{remark}

\subsection{Isogenies over $K(j)$: Part II}
We return to complete the classification of $K(j)$-rational cyclic isogenies for elliptic curves with CM by the orders of discriminants $\Delta = -4$ and $\Delta = -3$.  Recall that these cases have additional complexity 
coming from the fact that $\mu_K$ acts nontrivially on the projectivized torsion group $\mathbb{P} E[N]$.  In this case, 
there is an $\OO$-CM elliptic curve $(E_0)_{/K}$ for which the projective mod $N$ Galois representation
\[\mathbb{P} \rho_N \colon \gg_K \ra C_N(\OO)/(\Z/N\Z)^{\times}\] 
is surjective.  As we vary over the $K$-models of $E_0$, the representation $\mathbb{P} \rho_N$ twists 
by a character 
\[ \PP \chi \colon \gg_K \ra \mu_K/\{\pm 1\}. \]
Thus the index of $\mathbb{P} \rho_N(\gg_K)$ in $C_N(\OO)/(\Z/N\Z)^{\times}$ divides $2$ when $w_K = 4$ and divides 
$3$ when $w_K = 6$.
\\ \indent
We will rule out the existence of $K$-rational cyclic $N$-isogenies for various values of $N$ using the following ``$\widetilde{T}$-argument'': 
suppose that $\widetilde{T}(\OO,N) > \varphi(N) \frac{w_K}{2}$.  Then every $C_N(\OO)$-orbit on a point of order $N$ in $\OO/N\OO$ 
has size a multiple of $\widetilde{T}(\OO,N)$, so every $C_N(\OO)/(\Z/N\Z)^{\times}$-orbit on $\mathbb{P} E[N]$ has size a multiple of 
$\frac{\widetilde{T}(\OO,N)}{\varphi(N)}$, which by our hypothesis is greater than $\frac{w_K}{2}$.  So after passing to a field extension $L$ of degree $\frac{w_K}{2}$ to trivialize $\PP \chi$, we find that $\gg_L$ acts without fixed points on $\mathbb{P} E[N]$, and there 
is no $L$-rational cyclic $N$-isogeny and thus certainly no $K$-rational cyclic $N$-isogeny.   
\\ \\
Let $\OO$ be the order of discriminant $\Delta = -4$, so $K(j) = K = \Q(\sqrt{-1})$ and $w_K = 4$.  \\
$\bullet$ If $\ell \equiv 1 \pmod{4}$, then for all $a \in \Z^+$ we have that $-4$ is a square in $\Z/4\ell^a\Z$ so there is a $K$-rational cyclic $\ell^a$-isogeny.  In fact we get that \emph{every} $\OO$-CM elliptic curve $E_{/K}$ has a $K$-rational 
cyclic $\ell^a$-isogeny.   \\
$\bullet$ If $\ell \equiv 3 \pmod{4}$, since $\frac{\widetilde{T}(\OO,\ell)}{\varphi(\ell)\frac{w_K}{2}} =  \frac{\ell^2-1}{2 (\ell-1)} = \frac{\ell+1}{2} > 1$, by the $\widetilde{T}$-argument there is no $K$-rational $\ell$-isogeny.\\
$\bullet$ If $\ell = 2$, then since $T(\OO,4) = 1$, we can have a $K$-rational point of order $4$ (as already seen in $\S$7.5), hence 
a cyclic $K$-rational $4$-isogeny.  Since $\frac{\widetilde{T}(\OO,8)}{\varphi(8)\frac{w_K}{2}} = \frac{16}{4 \cdot 2} > 1$,
by the $\widetilde{T}$-argument there is no cyclic $K$-rational $8$-isogeny. \\
Any elliptic curve over a number field admitting a rational cyclic $N$-isogeny also admits a rational cyclic $M$-isogeny for all 
$M \mid N$.  Moreover, if an elliptic curve $E_{/F}$ admits $F$-rational cyclic $N_1,\ldots,N_r$ isogenies for pairwise coprime 
$N_1,\ldots,N_r$, then the subgroup generated by the kernels of these isogenies is $F$-rational and 
cyclic of order $N_1 \cdots N_r$ so $E$ admits an $F$-rational cyclic $N_1 \cdots N_r$-isogeny.  The assertion of 
Theorem \ref{KJISOGTHM}b) now follows.
\\ \\
Let $\OO$ be the order of discriminant $\Delta = -3$, so $K(j) = K = \Q(\sqrt{-3})$ and $w_K = 6$.  \\
$\bullet$ If $\ell \equiv 1 \pmod{3}$, then similarly to the $\Delta = -4$ case above we get that every $\OO$-CM elliptic curve $E_{/K}$ has a
$K$-rational cyclic $\ell^a$-isogeny for all $a \in \Z^+$. \\
$\bullet$ If $\ell \equiv 2 \pmod{3}$ and $\ell > 2$, then since $\frac{\widetilde{T}(\OO,\ell)}{\varphi(\ell) \frac{w_K}{2}} = \frac{\ell^2-1}{3(\ell-1)} = \frac{\ell+1}{3} > 1$,
by the $\widetilde{T}$-argument there is no cyclic $K$-rational $\ell$-isogeny.  \\
$\bullet$ If $\ell = 2$, then since $T(\OO,2) = 1$ there is an $\OO$-CM elliptic curve $E_{/K}$ with a $K$-rational $2$-isogeny.  \\
$\bullet$.  Since $\frac{\widetilde{T}(\OO,4)}{\varphi(4) \frac{w_K}{2}} = \frac{12}{2 \cdot 3} > 1$,
by the $\widetilde{T}$-argument there is no cyclic $K$-rational $4$-isogeny.  \\
$\bullet$ We claim that there is an $\OO$-CM elliptic curve $E_{/K}$ with a $K$-rational cyclic $9$-isogeny.  Let $\pp$ be 
the unique prime ideal of $\OO$ lying over $3$, and let $P$ be a generator of the cyclic $\OO$-module $E[\pp^3] \subset E[9]$, 
so $P$ has order $9$.  By Lemma \ref{7.1}, the $C_9(\OO)$-orbit on $P$ can be identified with the unit group $(\OO/\pp^3)^{\times}$, 
of order $18$.  The $\OO$-module generated by $P$ is also isomorphic to $(\zeta_3-1)\OO/9\OO$, and using this representation 
it is easy to compute that the group $(\OO/\pp^3)^{\times}$ is generated by the images of the scalar matrices $(\Z/9\Z)^{\times}$ 
and the cube roots of unity.  Thus Galois acts on the image of $P$ in $\mathbb{P} E[9]$ via a character $\mathbb{P} \chi$.  After twisting by the inverse of this character, the image of $P$ in $\mathbb{P} E[9]$ becomes fixed by Galois and we get a $K$-rational cyclic $9$-isogeny. \\
$\bullet$ Since $\frac{\widetilde{T}(\OO,18)}{\varphi(18) \frac{w_K}{2}} = \frac{54}{3 \cdot 6} > 1$, by the $\widetilde{T}$-argument 
there is no $K$-rational cyclic $18$-isogeny. \\
$\bullet$ Since $\frac{\widetilde{T}(\OO,27)}{\varphi(27) \frac{w_K}{2}} = \frac{162}{3 \cdot 18} > 1$,
by the $\widetilde{T}$-argument there is no $K$-rational cyclic $27$-isogeny.  \\
\\ 
$\bullet$ From $\S$7.5 (or Theorem \ref{BIGONE}) we know there is an $\OO$-CM elliptic curve $E_{/K}$ with a rational point of 
order $6$, hence certainly a cyclic $K$-rational $6$-isogeny.  \\
Using the same considerations as in the $\Delta = -4$ case above we get the assertion of Theorem \ref{KJISOGTHM}c).

\begin{example}
There are $13$ imaginary quadratic discriminants $\Delta$ such that the corresponding order $\OO(\Delta)$ has class number one.  For each such $\Delta$ we list the set of $N > 1$ for which there is an $\OO(\Delta)$-CM elliptic curve $E$ defined over $K = \Q(\sqrt{\Delta})$ that admits a $K$-rational cyclic $N$-isogeny -- otherwise put, for which there is an $\OO(\Delta)$-CM point on $X_0(N)(\Q(\sqrt{\Delta}))$.  
{\footnotesize
\begin{center}
    \begin{tabular}{ c|c}
$\Delta$ & Values of $N > 1$ with an $\OO(\Delta)$-CM point on $X_0(N)(\Q(\sqrt{\Delta}))$ \\ \hline
$-3$ & $2^a 3^b \ell_1^{a_1} \cdots \ell_r^{a_r}$ with $(a,b) \in \{(0,0), \ (1,0), \ (2,0), \ (0,1), \ (1,1)\}$, $\ell_i \equiv 1 \pmod{3}$ \\
$-4$ & $2^a \ell_1^{a_1} \cdots \ell_r^{a_r}$ with $a \leq 2$, $\ell_i \equiv 1 \pmod{4}$ \\ 
$-7$ &  $7^a \ell_1^{a_1} \cdots \ell_r^{a_r}$ with $a \leq 1$, $\left( \frac{\ell_i}{7} \right) = 1$ \\
$-8$ &  $2^a \ell_1^{a_1} \cdots \ell_r^{a_r}$ with $a \leq 1$, $\ell_i \equiv 1,3 \pmod{8}$ \\
$-11$ &  $11^a \ell_1^{a_1} \cdots \ell_r^{a_r}$ with $a \leq 1$, $\left( \frac{\ell_i}{11} \right) = 1$ \\
$-12$ &  $2^a 3^b \ell_1^{a_1} \cdots \ell_r^{a_r}$ with $a \leq 2$, $b \leq 1$, $\ell_i \equiv 1 \pmod{3}$ \\
$-16$ & $2^a \ell_1^{a_1} \cdots \ell_r^{a_r}$ with $a \leq 3$, $\ell_i \equiv 1 \pmod{4}$ \\
$-19$ & $19^a \ell_1^{a_1} \cdots \ell_r^{a_r}$ with $a \leq 1$, $\left( \frac{\ell_i}{19} \right) = 1$ \\
$-27$ & $2^a 3^b \ell_1^{a_1} \cdots \ell_r^{a_r}$ with $a \leq 2$, $b \leq 3$, $\ell_i \equiv 1 \pmod{3}$ \\
$-28$ &  $7^a \ell_1^{a_1} \cdots \ell_r^{a_r}$ with $a \leq 1$, $\left( \frac{\ell_i}{7} \right) = 1$ \\
$-43$ &  $43^a \ell_1^{a_1} \cdots \ell_r^{a_r}$ with $a \leq 1$, $\left( \frac{\ell_i}{43} \right) = 1$ \\
$-67$ & $67^a \ell_1^{a_1} \cdots \ell_r^{a_r}$ with $a \leq 1$, $\left( \frac{\ell_i}{67} \right) = 1$ \\
$-163$ & $163^a \ell_1^{a_1} \cdots \ell_r^{a_r}$ with $a \leq 1$, $\left( \frac{\ell_i}{163} \right) = 1$ 
     \end{tabular}
\end{center}
}
\end{example}


\begin{thebibliography}{CCRS14}

\bibitem[Ao95]{Aoki95} N. Aoki, \emph{Torsion points on abelian varieties with complex multiplication}.
    Algebraic cycles and related topics (Kitasakado, 1994), 1--22, World Sci. Publ., River Edge, NJ, 1995.

\bibitem[Ao06]{Aoki06} N. Aoki, \emph{Torsion points on CM abelian varieties}. Comment. Math. Univ. St. Pauli 55 (2006), 207--229.

\bibitem[BCII]{BCII} A. Bourdon and P.L. Clark, \emph{Torsion points and isogenies on CM elliptic curves}.  
\url{http://alpha.math.uga.edu/~pete/BCII.pdf}.

\bibitem[BCP17]{BCP} A. Bourdon, P.L. Clark and P. Pollack, \emph{Anatomy of torsion in the CM case}. Math. Z. 285 (2017), 795--820.

\bibitem[BCS17]{BCS} A. Bourdon, P.L. Clark and J. Stankewicz, \emph{Torsion points on CM elliptic curves over real number 
fields}.  Trans. Amer. Math. Soc. 369 (2017), 8457-–8496.


\bibitem[BP16]{BP16} A. Bourdon and P. Pollack, \emph{Torsion subgroups of CM elliptic curves over odd degree number fields}.   Int. Math. Res. Not. IMRN 2017,  4923--4961.

\bibitem[Br10]{Breuer10} F. Breuer, \emph{Torsion bounds for elliptic curves and Drinfeld modules}. J. Number Theory 130 (2010), 1241--1250.


\bibitem[CA]{Clark-CA} P.L. Clark, \emph{Commutative Algebra}. \url{http://math.uga.edu/~pete/integral2015.pdf}

\bibitem[CCS13]{TORS1} P.L. Clark, B. Cook and J. Stankewicz, \emph{Torsion points on elliptic curves with complex multiplication (with an appendix by Alex Rice)}.  International Journal of Number Theory 9 (2013), 447--479.

 \bibitem[CCRS14]{TORS2} P.L. Clark, P. Corn, A. Rice and J. Stankewicz, \emph{Computation on Elliptic Curves with Complex Multiplication}. LMS J. Comput. Math. 17 (2014), no. 1, 509--535.

\bibitem[Co00]{Cohen2} H. Cohen, \emph{Advanced Topics in Computational Number Theory}.  Graduate Texts in  Mathematics 193, Springer-Verlag, 2000.

\bibitem[Co89]{Cox89} D. Cox, \emph{Primes of the form $x^2+ny^2$.  Fermat, class field theory and complex multiplication.}  John Wiley $\&$ Sons, New York, 1989.


\bibitem[CP15]{CP15} P.L. Clark and P. Pollack, \emph{The truth about torsion in the CM case.}
C. R. Math. Acad. Sci. Paris 353 (2015), 683-–688. 

\bibitem[Ei]{Eisenbud} D. Eisenbud,
\emph{Commutative  Algebra.} Graduate Texts in Mathematics 150, Springer-
Verlag, 1995.

\bibitem[Fr35]{Franz35} W. Franz, \emph{Die Teilwerte der Weberschen Tau-Funktion}.
J. Reine Angew. Math. 173 (1935), 60–-64. 

\bibitem[GJ19]{GJ19} E. Gonz\'alez-Jim\'enez, \emph{Explicit characterization of the torsion growth of rational 
elliptic curves with complex multiplication over quadratic fields}, preprint.


\bibitem[GR18]{Gaudron-Remond18} \'E. Gaudron and G. R\'emond, \emph{Torsion des vari\'et\'es ab\'eliennes CM.} Proc. Amer. Math. Soc. 146 (2018),  2741--2747.

\bibitem[HW08]{HW} G.H. Hardy and E.M. Wright, \emph{An introduction to the theory of numbers}. Sixth edition.
Oxford University Press, Oxford, 2008.

\bibitem[Kw99]{Kwon99} S. Kwon, \emph{Degree of isogenies of elliptic curves with complex multiplication}. J. Korean Math. Soc. 36 
(1999), 945--958. 

\bibitem[La87]{LangEll} S. Lang, \emph{Elliptic functions. With an appendix by J.Tate. Second edition.} Graduate Texts in Mathematics, 112. Springer-Verlag, New York, 1987.

\bibitem[Lo15]{Lombardo15} D. Lombardo, \emph{Galois representations attached to abelian varieties of CM type}. 
Bull. Soc. Math. France 145 (2017),  469–-501.

\bibitem[LR18]{LR} \'{A}. Lozano-Robledo, \emph{Uniform boundedness in terms of ramification}. 
Res. Number Theory 4 (2018), no. 1, Art. 6, 39 pp.

\bibitem[LR19]{LR19} \'{A}. Lozano-Robledo, \emph{Galois representations attached to elliptic curves with 
complex multiplication}.  \url{https://arxiv.org/abs/1809.02584}

\bibitem[Ma77]{Mazur77} B. Mazur, \emph{Rational points on modular curves.} Modular functions of one variable, V (Proc. Second Internat. Conf., Univ. Bonn, Bonn, 1976), pp. 107–148. Lecture Notes in Math., Vol. 601, Springer, Berlin, 1977. 

\bibitem[Na10]{Najman10} F. Najman, \emph{Complete classification of torsion of elliptic curves over quadratic cyclotomic fields}.
J. Number Theory 130 (2010), 1964--1968.

\bibitem[Na11]{Najman11} F. Najman, \emph{Torsion of elliptic curves over quadratic cyclotomic fields.} Math. J. Okayama Univ. 53 (2011), 75--82. 

 \bibitem[Ol74]{Olson74} L. Olson, \emph{Points of finite order on elliptic curves with complex multiplication}. Manuscripta Math. 14 (1974), 195--205.

 \bibitem[Pa89]{Parish89} J.L. Parish, \emph{Rational Torsion in Complex-Multiplication Elliptic Curves}.  Journal of Number Theory 33 (1989), 257--265.

 \bibitem[PY01]{PY01} D. Prasad and C.S. Yogananda,
 \emph{Bounding the torsion in CM elliptic curves}.  C. R. Math. Acad. Sci. Soc. R. Can.  23  (2001), 1--5.

\bibitem[Ro94]{Ross94} R. Ross, \emph{Minimal torsion in isogeny classes of elliptic curves} Trans. Amer. Math. Soc. 344 (1994), 203--215.

\bibitem[S72]{Serre72} J.-P. Serre, \emph{Propri\'et\'es galoisiennes des points d'ordre fini des courbes
elliptiques}. Invent. Math.  15  (1972),  259--331.

\bibitem[Si88]{Silverberg88} A. Silverberg, \emph{Torsion points on abelian varieties of CM-type}.  Compositio Math.  68  (1988),  241--249.

\bibitem[Si92]{Silverberg92} A. Silverberg,  \emph{Points
of finite order on abelian varieties}.  In \emph{$p$-adic methods
in number theory and algebraic geometry}, 175--193,
Contemp. Math. 133, Amer. Math. Soc., Providence, RI, 1992.


\bibitem[Si94]{SilvermanII} J. Silverman, \emph{Advanced Topics in the Arithmetic of Elliptic Curves}, Graduate Texts in
Mathematics 151, Springer-Verlag, 1994.


\bibitem[St01]{Stevenhagen01} P. Stevenhagen, \emph{Hilbert's 12th problem, complex multiplication and Shimura reciprocity.} Class field theory -- its centenary and prospect (Tokyo, 1998), 161–-176,
Adv. Stud. Pure Math., 30, Math. Soc. Japan, Tokyo, 2001.

\end{thebibliography}
\end{document}